\documentclass{amsart}

\usepackage[normalem]{ulem}

\usepackage{float}

\usepackage{lineno}

\usepackage{graphicx}

\usepackage{soul}

\usepackage{amssymb, amsmath, amsthm, hyperref, euscript, enumerate, color}
\usepackage[dvipsnames]{xcolor}

\usepackage{amscd}

\usepackage[shortlabels]{enumitem}

\usepackage{bbm}

\usepackage{subcaption}

\usepackage{mathtools}
\usepackage{tikz-cd}

\usepackage[english]{babel}
\numberwithin{equation}{section}

\theoremstyle{plain}

\newtheorem{theorem}{Theorem}

\newtheorem{proposition}[theorem]{Proposition}

\newtheorem{remark}[theorem]{Remark}

\newtheorem{lemma}[theorem]{Lemma}

\newtheorem{thm}{Theorem}

\newtheorem{qst}[thm]{Question}

\theoremstyle{definition}
\newtheorem{definition}[theorem]{Definition}

\newcommand{\R}{\mathbb{R}}

\newcommand{\Z}{\mathbb{Z}}

\newcommand{\N}{\mathbb{N}}

\makeatletter
\@namedef{subjclassname@2020}{%
  \textup{2020} Mathematics Subject Classification}
\makeatother

\begin{document}

\author{Valentina Bais, Younes Benyahia, Oliviero Malech and Rafael Torres}

\title[(Un)linked 2-links in 4-manifolds and knot surgery]{A recipe for exotic 2-links in closed 4-manifolds whose components are topological unknots}

\address{Scuola Internazionale Superiore di Studi Avanzati (SISSA)\\ Via Bonomea 265\\34136\\Trieste\\Italy}

\email{$\{vbais, ybenyahi, omalech, rtorres\}$@sissa.it}


\maketitle

\emph{Abstract}: We describe a construction procedure of infinite sets of $2$-links in closed simply connected 4-manifolds that are topologically isotopic, smoothly inequivalent and componentwise topologically unknotted. These 2-links are the first examples of such kind in the literature. The examples provided have surface and free groups as their 2-link groups. We also point out an exotic Brunnian behaviour of such families, which highlights the important role of linking in creating exotic phenomena.

\section{Introduction and main results.}\label{Introduction}

All embeddings and manifolds in this paper are smooth unless it is stated otherwise. The main contribution of this paper is to introduce a construction procedure of infinite sets of multiple components 2-links in closed 4-manifolds that are topologically isotopic, yet smoothly inequivalent, and whose components are topologically unknotted. The explicit examples that we produce have the property that the fundamental group of the complement can be chosen to be any free group or a surface group. We begin by making these notions precise in the following definition. 


\begin{definition}\label{Definition Knotted Link}$\bullet$ A $k$-component $2$-link
\begin{equation}\label{Definition Link}\Gamma = S_1 \sqcup \cdots \sqcup S_k\subset X \end{equation}
in a $4$-manifold $X$ is an unordered union of disjointly embedded $2$-spheres $S_i$ with trivial tubular neighborhood $\nu(S_i) = D^2\times S^2$ for $i=1,\dots,k$ and $k\in \N$. We say that a component $S_i$ is topologically unknotted if there is a locally flat embedded 3-ball $D_i\subset X$ such that $\partial D_i = S_i$.

$\bullet$ The fundamental group $\pi_1(X\setminus \Gamma)$ of the complement of (\ref{Definition Link}) is called the 2-link group.

$\bullet$ A 2-link (\ref{Definition Link}) is symmetric if for any permutation $\sigma$ of $\{1, \ldots, k\}$ there is a diffeomorphism $\phi: X\rightarrow X$ such that $\phi(S_i) = S_{\sigma(i)}$ for $i = 1, \ldots, k$.

$\bullet$ Two $k$-component $2$-links $\Gamma$ and $\Gamma'$ embedded in a $4$-manifold $X$ are smoothly inequivalent if there is no diffeomorphism of pairs\begin{equation*}(X, \Gamma)\rightarrow (X, \Gamma').\end{equation*} 

$\bullet$ A family of $k$-component $2$-links is exotic if its elements are topologically isotopic and pairwise smoothly inequivalent.

\end{definition}

There has been a flurry of research on exotic embeddings of surfaces in closed 4-manifolds ignited by Fintushel-Stern's foundational paper \cite{[FintushelSternS]}. Finding such embeddings of nullhomotopic 2-spheres is particularly difficult:  just notice that it is still unknown if there is an exotic $S^2$ in $S^4$. There is an implicit topologically unknotted 2-sphere that is smoothly knotted in a closed simply connected $4$-manifold in work of Fintushel-Stern \cite{[FintushelStern1]} as pointed out by Ray-Ruberman in \cite{[RayRuberman]}, and infinite sets of such embeddings of 2-spheres were constructed by the fourth author of this note in \cite[Theorem A]{[Torres]}. 





The construction procedure we introduce in this paper extends these results to multiple components 2-links as we now sample with our main result. A closed oriented surface of genus $g$ is denoted by $\Sigma_g$, and the free group of $g$ generators by $F_g$.

\begin{thm}\label{Theorem A} Let $M$ be an admissible $4$-manifold as in Definition \ref{Definition Admissible Manifold} and fix $g\in \N$. Let $G$ be a group isomorphic to either $\pi_1(\Sigma_g)$ or $F_g$, and let $n$ be its rank. Let $\mathcal{K}$ be an infinite family of knots $K \subset S^3$ with pairwise distinct Alexander polynomials.
\\ There is an infinite set of smooth $n$-component 2-links\begin{equation}\label{Unlink Main}\{\Gamma_{K} = S_{1, K} \sqcup \cdots \sqcup S_{n, K}\mid K\in \mathcal{K} \}\end{equation} 
which are
smoothly embedded in $M\#n(S^2\times S^2)$ and have the following properties.\begin{enumerate}

\item Their 2-link group is $G$.
\item The 2-links in (\ref{Unlink Main}) form an exotic family as in Definition \ref{Definition Knotted Link}.
\item The components of (\ref{Unlink Main}) are topologically unknotted.
\item Surgery on the components of (\ref{Unlink Main}) yields an infinite set of pairwise non-diffeomorphic closed 4-manifolds in a same homeomorphism class with fundamental group $G$.

\end{enumerate}
\end{thm}

 Theorem \ref{Theorem A} unveils exotic phenomena of codimension two embeddings that were not previously known to exist. A detailed description of the procedure is given in Section \ref{Section Strategy}. Examples of exotic 2-links in closed 4-manifolds with homologically essential components are available. Auckly-Kim-Melvin-Ruberman \cite[Theorem B]{[AucklyKimMelvinRuberman]} produced infinite sets of exotic 2-links with trivial 2-link group and whose components have self-intersection +1. Hayden-Kjuchukova-Krishna-Miller-Powell-Sunukjian provided pairs of exotic $2$-component $2$-links in \cite[Theorem 8.2]{[HKKMPS]}; see \cite{[Hayden]} as well. 

Much like the examples in \cite{[AucklyKimMelvinRuberman], [HKKMPS]}, linking is of significant importance in the exotic behaviour of the $2$-links of Theorem \ref{Theorem A}. In order to understand this behaviour better, we introduce the following notion. 

\begin{definition}\label{Definition Brunnianly}An exotic pair of smooth $k$-component 2-links\begin{center}$\Gamma= S_1 \sqcup \dots \sqcup S_k$ and $\Gamma'=S_1' \sqcup \dots \sqcup S_k'$\end{center} in the sense of Definition \ref{Definition Knotted Link} is Brunnianly exotic if the $2$-links\begin{center}$\Gamma \setminus S_i$ and $ \Gamma' \setminus S_j'$\end{center} are smoothly equivalent for any $i,j = 1, \dots,k$.
\end{definition}

The Brunnianity property of Definition \ref{Definition Brunnianly} rules out uninteresting constructions of exotic 2-links with free 2-link group as described in Remark \ref{Remark Trivial Situation}. Our second main result shows that the examples of Theorem \ref{Theorem A} satisfy this property and that they stabilize by taking the connected sum with a single copy of $S^2\times S^2$ at a point away from every element in (\ref{Unlink Main}).

\begin{thm}\label{Theorem B}\

$\bullet$ Every 2-link in the infinite set (\ref{Unlink Main}) of Theorem \ref{Theorem A} with free $2$-link group is smoothly symmetric.
    
   $\bullet$ There is an infinite subset of (\ref{Unlink Main}) made of pairwise Brunnianly exotic $2$-links smoothly embedded in $M\# g (S^2\times S^2)$. Moreover, elements in this infinite subset are pairwise smoothly equivalent in $(M\# g(S^2\times S^2))\# (S^2\times S^2)$.
    
\end{thm}

Related notions of Brunnianity of 2-links in 4-manifolds are available in the literature, always in analogy with the properties of Brunnian links in $S^3$ introduced by Brunn in 1892 \cite{[Bru92]}. The notion of Bruniannity used by Auckly-Kim-Melvin-Ruberman in \cite[Theorem B]{[AucklyKimMelvinRuberman]} is quite similar to ours albeit a bit stronger: their results address smooth isotopy by considering ordered links. We consider only equivalence and not isotopy between our surfaces, but do so irrespectively of their order. A stronger use of the adjective Brunnian is employed by Hayden-Kjuchukova-Krishna-Miller-Powell-Sunukjian in \cite{[HKKMPS]}.

We finish the introduction with a pair of questions that arose during the production of this work. The first question concerns a topological property of the 2-links constructed under our procedure. Notice that the 2-links (\ref{Unlink Main}) of Theorem \ref{Theorem A} with $F_g$ as 2-link group are smoothly linked: there are no $g$ smoothly embedded disjoint 3-balls in $M\#g(S^2\times S^2)$ such that each of them bounds a component of $\Gamma_K$. The locally flat embbeded realm imediately comes to mind. 

\begin{qst} Are the 2-links of Theorem \ref{Theorem A} with free 2-link group topologically unlinked? That is, are there locally flat embedded disjoint $3$-balls\begin{equation}\label{3-balls}\Theta_g = D^3_1\sqcup \cdots \sqcup D^3_g\end{equation}in $M\#g(S^2\times S^2)$ such that $\partial \Theta_K = \Gamma_k$ for any $K\in \mathcal{K}$? 
\end{qst}

The second question regards a comparison of the Brunnianity property of Definition \ref{Definition Brunnianly} and the Brunnian behaviour of the examples of Hayden-Kjuchukova-Krishna-Miller-Powell-Sunukjian.

\begin{qst} Are the 2-links of Theorem \ref{Theorem A} with free 2-link group Brunnian in the sense of \cite{[HKKMPS]}? That is, is $\Gamma_K \setminus S_{i,K}$ smoothly unlinked for every $i=1, \dots, g$?
\end{qst}






The paper is organized as follows. The construction procedure of infinite sets of exotic 2-links is laid down in Section \ref{Section Strategy}. A topological restriction on admissible 4-manifolds is given in Section \ref{Section Admissible Manifolds}. The main building blocks are produced in Section \ref{Section manifold W_g}, Section \ref{Section Kodaira Thurston} and Section \ref{Section Manifold N_g}. These sections contain results to pin down the diffeomorphism type of a 4-manifold constructed from surgeries that might be of independent interest, including a handlebody depiction of the Kodaira-Thurston manifold not previously available in the literature to the best of our knowledge. The smooth structures of Theorem \ref{Theorem A} are constructed in Section \ref{Section Smooth Structures}. The diffeomorphism type of the ambient 4-manifolds are pinned down in Section \ref{Section Diffeomorphism Type}. The 2-links are manufactured in Section \ref{Section TOPIsotopy}, where we also show that they are pairwise topologically isotopic. The properties of Definition \ref{Definition Brunnianly} are investigated in Section \ref{Section Brunnian}.

\subsection{Acknowledgements}
The first author is a member of GNSAGA – Istituto Nazionale di Alta Matematica ‘Francesco Severi’, Italy.

\section{Recipe for exotic 2-links.}\label{Strategies}

\subsection{Admissible 4-manifolds of Theorem \ref{Theorem A} and their size.}\label{Section Admissible Manifolds} The 4-manifolds that are our key raw material of our construction procedure and the hypothesis of Theorem \ref{Theorem A} are defined as follows. 

\begin{definition}\label{Definition Admissible Manifold}An admissible $4$-manifold $M$ is a $4$-manifold that satisfies the following properties.
\begin{itemize}\item There is at least one basic class $k\in H^2(M;\Z)$, i.e. the Seiberg-Witten invariant of $M$ satisfies $SW_M(k)\neq 0$.

\item There is a pair of embedded disjoint $2$-tori $T_1,T_2\subset M$ such that\begin{center}$T_i$ has self-intersection $0$ for $i = 1, 2$,\end{center} and \begin{center}$\pi_1(M) = \{1\} = \pi_1(M\setminus (\nu(T_1)\sqcup \nu(T_2))$.\end{center}

\end{itemize}

\end{definition}

There is no shortage of admissible $4$-manifolds in the literature \cite{[BaldridgeKirk2], [FintushelStern1], [Gompf2], [McCarthyWolfson]}. The elliptic surfaces $E(n)$ for $n\geq 2$ are a prototype example, but admissibe 4-manifolds with smaller second Betti number are known. In \cite[Theorem 18]{[BaldridgeKirk2]}, Baldridge-Kirk construct an admissible 4-manifold that is homeomorphic to $3\mathbb{CP}^2\#5\overline{\mathbb{CP}^2}$. The presence of the 2-tori $T_1$ and $T_2$ of Definition \ref{Definition Admissible Manifold} imposes certain restrictions on the second Betti number of an admissible 4-manifold, which are recorded in the following inequality (cf. \cite[Lemma 1]{[FintushelStern]}).

 \begin{lemma}\label{lemma inequality Betty number M}
     If $M$ is an admissible 4-manifold in the sense of Definition \ref{Definition Admissible Manifold}, then 
     \begin{equation}\label{eq: inequality betty number of M}
         b_2(M)\geq |\sigma(M)|+4 
     \end{equation} where $b_2(M)$ and $\sigma(M)$ are the second Betti number and the signature of $M$, respectively. In particular, $M$ has an indefinite intersection form. 
 \end{lemma}

\begin{proof}Since the complement $M \setminus (\nu(T_1) \sqcup \nu(T_2))$ is simply connected, each 2-torus $T_i$ has a geometrically dual surface $S_i$ for $i = 1, 2$. We define the linear subspace\begin{equation*}\mathcal{S}=\langle [T_1],[S_1],[T_2],[S_2]\rangle \subset H_2(M;\R)\end{equation*}
     and compute the intersection form on its generators to be the matrix
     \begin{equation*}\label{Matrix 1} \begin{bmatrix}
0 & 1 & 0 & 0 \\
1 & * & 0 & * \\
0 & 0 & 0 & 1 \\
0 & * & 1 & * 
\end{bmatrix}.  \end{equation*}
The matrix (\ref{Matrix 1}) is equivalent to 
$A=2\begin{pmatrix}
+1
\end{pmatrix} \oplus 2 \begin{pmatrix}
-1
\end{pmatrix}  $
via a real change of basis.
We can extend this new basis to a basis for $H_2(M;\R)$, so that, after possibly changing again coordinates, the intersection form $Q_M$ over $\R$ becomes
\begin{equation*}Q_M\cong  \begin{bmatrix}
A & 0 \\
0 & B 
\end{bmatrix} .
\end{equation*}
This yields the desired inequality\begin{equation*}|\sigma(M) |=|\sigma(Q_M)|=|\sigma(B)|\leq rank(B)=b_2(M)-4.\end{equation*}

 \end{proof}

\begin{remark}[Non-empty boundary and the size of admissible 4-manifolds] Hayden-Kjuchukova-Krishna-Miller-Powell-Sunukjian in \cite{[HKKMPS]} construct pairs of exotic properly embedded  2-disks in the 4-ball, where the adjective ‘exotic’ in this case means topologically isotopic rel. boundary and pairwise smoothly inequivalent (see \cite{[Hayden]} too). In order to distinguish the smooth structures of their complements, they make use of Stein structures and a well-known adjunction formula \cite[Theorem 11.4.7]{[GompfStipsicz]}; cf. \cite[Section 9.1]{[Akbulut2]}. In order to be able to distinguish infinite sets and not only pairs of exotic 2-links, we employ the Seiberg-Witten invariants instead and work with larger ambient 4-manifolds.
\end{remark}

\subsection{Strategy to construct exotic 2-links and prove Theorem \ref{Theorem A}}\label{Section Strategy} 

\begin{enumerate}[(a)]
\item Start with an admissible 4-manifold of Definition \ref{Definition Admissible Manifold} and perform Fintushel-Stern's knot surgery \cite{[Fintushel], [FintushelStern3]} along the 2-torus $T_1$ to produce an infinite set $\{M_K\mid K\in \mathcal{K}\}$ of pairwise non-diffeomorphic 4-manifolds that are homeomorphic to $M$. Notice that each one of these 4-manifolds contains a copy of the 2-torus $T_2$. 
\item\label{fiber} Build the generalized fiber sum \cite{[Gompf2]} $Z_{K} = M_K\#_{T^2} B_G$ of $M_K$ and a 4-manifold $B_G$. For an appropriate choice of the building block $B_G$, this yields an infinite set of pairwise non-diffeomorphic 4-manifolds in the homeomorphism type of $Z_{K}$ with fundamental group $G$. 
\item\label{surgeries} Perform surgery along $n$ loops $\mathcal{L}_G=\{\gamma_1, \dots, \gamma_n \}$ in $Z_{K}$ whose homotopy classes are the generators of the group $\pi_1(Z_{K}) = G$, and produce a closed simply connected 4-manifold $Z_{K}^\ast$ that is diffeomorphic to $M\#n(S^2\times S^2)$ for every $K\in \mathcal{K}$. This step can be traced back to Wallace \cite{[Wallace]} and Milnor \cite{[Milnor]}. 
\item The belt 2-spheres of the surgeries of Step (c) form a 2-link $\Gamma_{K}$ of $n$ nullhomotopic components smoothly embedded in $M\#n(S^2\times S^2)$ whose 2-link group is isomorphic to $G$. A result of Sunukjian \cite[Theorem 7.2]{[Sunukjian2]} implies that each component of $\Gamma_K$ is topologically unknotted.
\item Any two 2-links $\Gamma_{K}$ and $\Gamma_{K'}$ are smoothly inequivalent for any two different knots $K, K' \in \mathcal{K}$ since their complements\begin{equation}\label{Item St}M\#n(S^2\times S^2)\setminus \nu(\Gamma_K) \qquad \text{and}\qquad M\#n(S^2\times S^2)\setminus \nu(\Gamma_{K'})\end{equation} are non-diffeomorphic. This is proven indirectly using gauge theoretical invariants of the closed 4-manifolds $Z_K$ and $Z_{K'}$ of Item (b). 
\item The explicit nature of our constructions yields a homeomorphism of the complements (\ref{Item St}) that extends to a homeomorphism of pairs\begin{equation}(M\# n(S^2\times S^2), \Gamma_K)\rightarrow (M\# n(S^2\times S^2), \Gamma_{K'}),\end{equation}which induces the identity map on homology. Results of Quinn and Perron allows us to conclude that the $2$-links $\{\Gamma_K \mid K \in \mathcal{K} \}$ are topologically isotopic.
\item Surgery on one loop in \ref{surgeries}, instead of $n$ loops, is already enough to produce the same diffeomorphism type regardless of the knot $K$. This is used to show that giving up one componenent in each 2-link stabilizes the family (\ref{Unlink Main}).
\end{enumerate}

\section{Building blocks and auxiliary results.}\label{Section Results} 

We gather in this section the raw materials that are used in the proof of Theorem \ref{Theorem A} and record several of  their topological properties that are useful for our purposes. The following items are fixed in the sequel.
\begin{itemize}
    \item A natural number $g \in \N$. 
    \item A group $G \in \{ F_g, \pi_1(\Sigma_g) \}$ of rank $n$. 
    \item An infinite collection $\mathcal{K}$ of knots $K \subset S^3$ with pairwise distinct Alexander polynomials. For simplicity, we assume that the family $\mathcal{K}$ contains the unknot $U\subset S^3$. 
    \item An admissible manifold $M$ as in Definition \ref{Definition Admissible Manifold}.
\end{itemize}

\subsection{\label{Section manifold W_g} $T^2\times S^2\#2g(S^2\times S^2)$ as a result of surgery to $T^2\times \Sigma_g$}

We begin this section with a description of the building block $B_{\pi_1(\Sigma_g)} = T^2\times \Sigma_g$ (see step (b) of Section \ref{Section Strategy}), where $T^2$ is the 2-torus. To ease notation, we will call such building block $B_g$. 

We pick on the surfaces $T^2$ and $\Sigma_g$ simple closed curves $x,y$ and $a_i,b_i$ for $i=1,\dots,g$ as in Figure \ref{fig:T2xSigmag}. In particular, writing $T^2=S^1\times S^1$, the loop $x$ is $S^1\times \{1\}$ and the loop $y$ is $\{1\}\times S^1$. Moreover, we take as a parallel copy of $x$ the loop $x'=S^1\times \{-1\}$.  Notice that the loops $\{a_i,b_i \mid i=1, \dots,g \} $ form a symplectic basis for the first homology group $H_1(\Sigma_g; \mathbb{Z})$.

 \begin{figure}[h]
  \includegraphics[width= 0.8\textwidth]{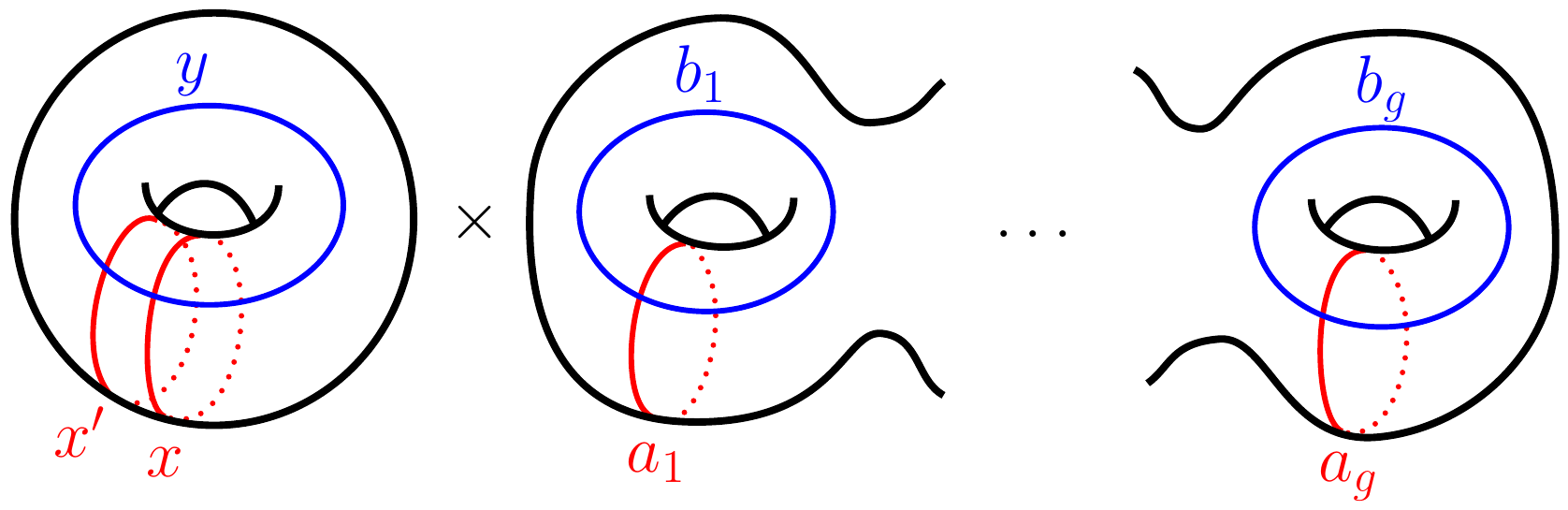} 
  \caption{The 4-manifold $T^2\times \Sigma_g$, the curves $x,x',y$ in $T^2$ and $a_i,b_i$ for $i=1,\dots,g$ in $\Sigma_g$. 
  }
  \label{fig:T2xSigmag}
\end{figure}
On the other hand, the surfaces $T^2\times \{p\}$, $\{p\}\times \Sigma_g$ and the collections of $2$-tori
\begin{equation}\label{SetToriA}
    \{x\times a_i\mid i=1,\dots,g\},
\end{equation}
\begin{equation}\label{SetToriB}
    \{x'\times b_i\mid i=1,\dots,g\},
\end{equation}and
\begin{equation}\label{SetToriC}
    \{y\times a_i\mid i=1,\dots,g\}\text{ and } \{y\times b_i\mid i=1,\dots,g\}
\end{equation}generate the second homology group $H_2(T^2 \times \Sigma_g; \Z ) = \Z^{2+4g}$. 

The second building block is a $4$-manifold $B_g^*$ that is obtained from $T^2\times \Sigma_g$ by doing loop surgery along the framed components of the $1$-dimensional submanifold
 \begin{equation}\label{Set of loops gamma}
 \mathcal{L}_{\pi_1(\Sigma_g)}=\gamma_1 \sqcup \gamma_1' \sqcup \ldots \sqcup \gamma_g \sqcup \gamma_g' \subset T^2\times \Sigma_g
 \end{equation} 
where $\gamma_i$ is the loop $\{p\}\times a_i$ and $\gamma_i' $ is $\{p'\}\times b_i$
 and the framing is the one induced by the product structure of $T^2 \times \Sigma_g$. 
 
 In particular, we get $B_g^*$ as \begin{equation}\label{Auxiliary Manifold 2}B^\ast_g = (T^2\times \Sigma_g) \setminus \bigsqcup_{i=1}^g (\nu(\gamma_i)\sqcup \nu(\gamma'_i))\cup \bigsqcup_{i=1}^{2g} (D^2\times S^2).\end{equation}
  Notice that doing surgery on the loops (\ref{Set of loops gamma}) kills the subgroup\begin{equation*}\{1\}\times \pi_1(\Sigma_g) <\pi_1(T^2\times \Sigma_g)=\pi_1(T^2)\times \pi_1(\Sigma_g).\end{equation*} 

 
We also build an auxiliary 4-manifold $\widehat{B}_g$ as the result of a performing a total of $2g$ torus surgeries on the framed 2-tori (\ref{SetToriA}) and (\ref{SetToriB}). More explicitly, this 4-manifold is defined as\begin{equation}\label{Auxiliary Manifold 1} \widehat{B}_g = (T^2\times \Sigma_g) \setminus (\bigsqcup_{i=1}^g (\nu(x\times a_i) \sqcup \nu(x' \times b_i))\cup_\phi \bigsqcup_{i=1}^{2g} (T^2\times D^2))\end{equation}where $\phi$ is a gluing diffeomorphism\begin{equation*}\phi:\bigsqcup_{i=1}^g\partial\nu(x\times a_i)\cup\partial\nu(x'\times b_i)\to \bigsqcup_{i=1}^{2g} (T^2\times \partial D^2)\end{equation*} that satisfies
\begin{equation*}
    (\phi|_{\partial\nu(x\times a_i)})^{-1}_*([\{p\}\times \partial D^2])=[l_i]
    \in H_1(\partial\nu(x\times a_i);\Z)
\end{equation*} and a similar equation for the 2-tori of the form $x'\times b_i$ for any $i=1,\dots,g$. Here $l_i$ is a Lagrangian push off of the loop $\gamma_i$ into $\partial\nu(x\times a_i)$. 

This cut-and-paste construction is known as a multiplicity zero log transform along the loops $\gamma_i$ and $\gamma_i'$ \cite[p. 83]{[GompfStipsicz]}. We now briefly explain a trick due to Moishezon \cite[Lemma 13]{[Moishezon]}, which will play a key role in Lemma \ref{Lemma Diffeomorphism 3} and Lemma \ref{Lemma Diffeomorphism 2}. In general, the transformation of a $T^2\times D^2$ under a multiplicity zero log transform is depicted in Items (A) and (B) in Figure \ref{fig:Argument}, where we use the dotted circle notation for 1-handles \cite{[Akbulut0]}, \cite{[CdeSa1], [CdeSa2]}, \cite[Section 5.4]{[GompfStipsicz]}. Item (A) depicts a copy of $T^2\times D^2$ inside a smooth 4-manifold $X$, where all handles that go through the two dotted circles and the 0-framed circle are ignored. Item (B) depicts what replaces the copy of $T^2\times D^2$ when the zero log transform is performed, yielding the manifold $\widehat{X}$ \cite[Figure 6.9]{[Akbulut2]}, \cite[Figure 8.25]{[GompfStipsicz]}; all other handles are ignored. Item (C) depicts the same area after performing a loop surgery on $X$, we call the result $X^*$. Notice that $X^*$ can also be obtained by performing a loop surgery on $\widehat{X}$. The trick is to split a loop surgery into a multiplicity zero log transform followed by a different loop surgery.


 \begin{figure}[h]
     \centering
     \begin{subfigure}[b]{0.26\textwidth}
         \centering
         \includegraphics[width=\textwidth]{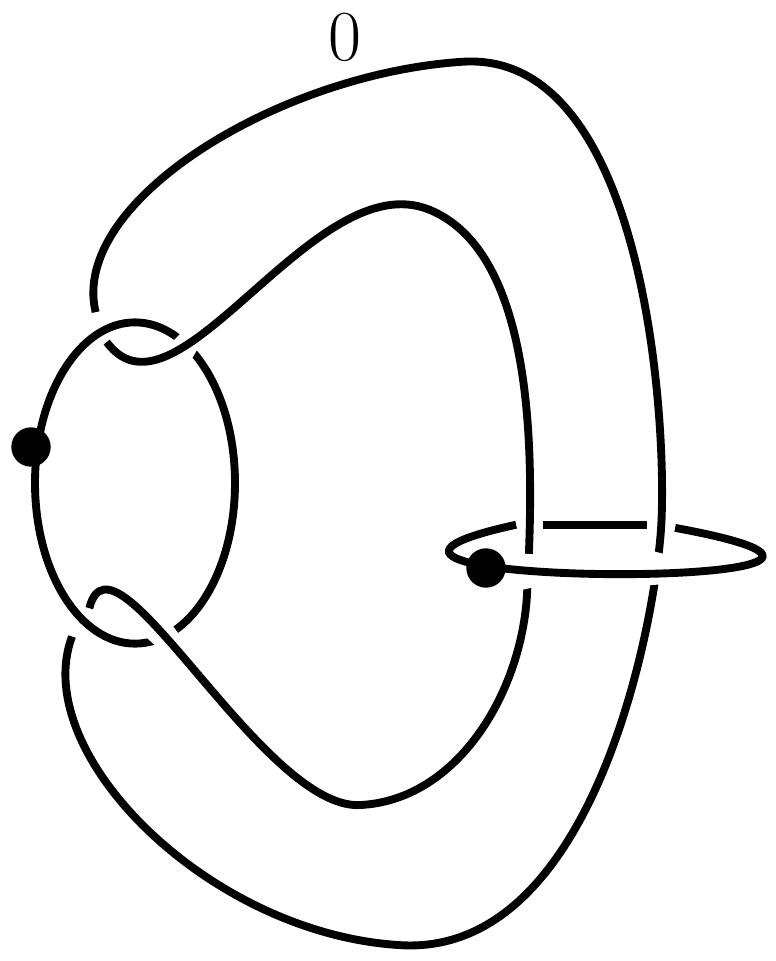}
         \caption{}
         \label{}
     \end{subfigure}
     \hfill
     \begin{subfigure}[b]{0.26\textwidth}
         \centering
         \includegraphics[width=\textwidth]{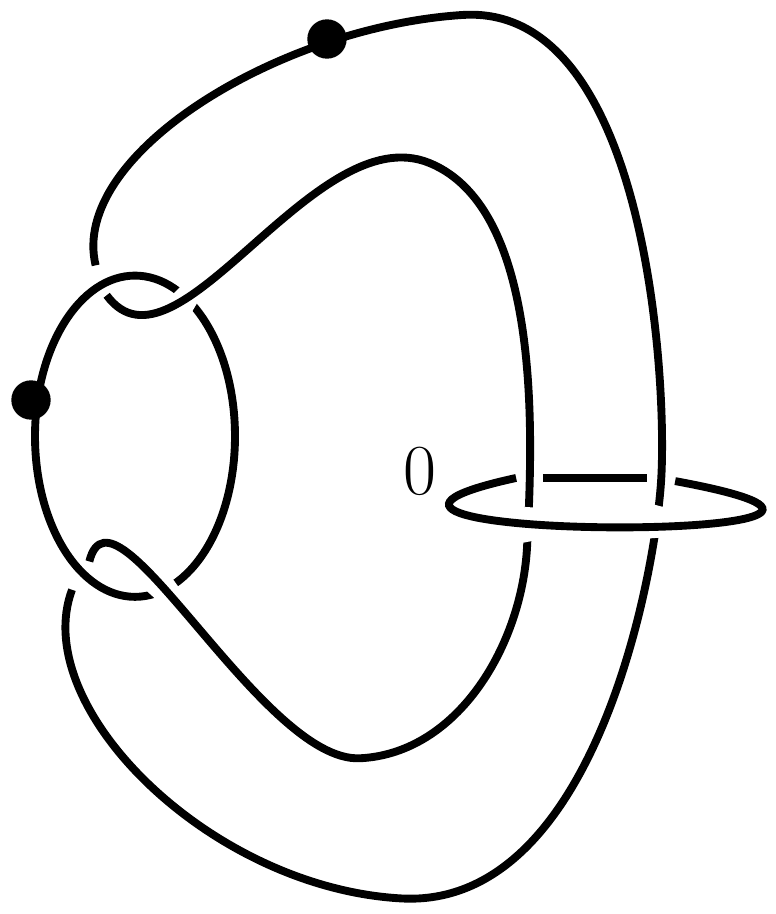}
         \caption{}
         \label{}
     \end{subfigure}
     \hfill
     \begin{subfigure}[b]{0.26\textwidth}
         \centering
         \includegraphics[width=\textwidth]{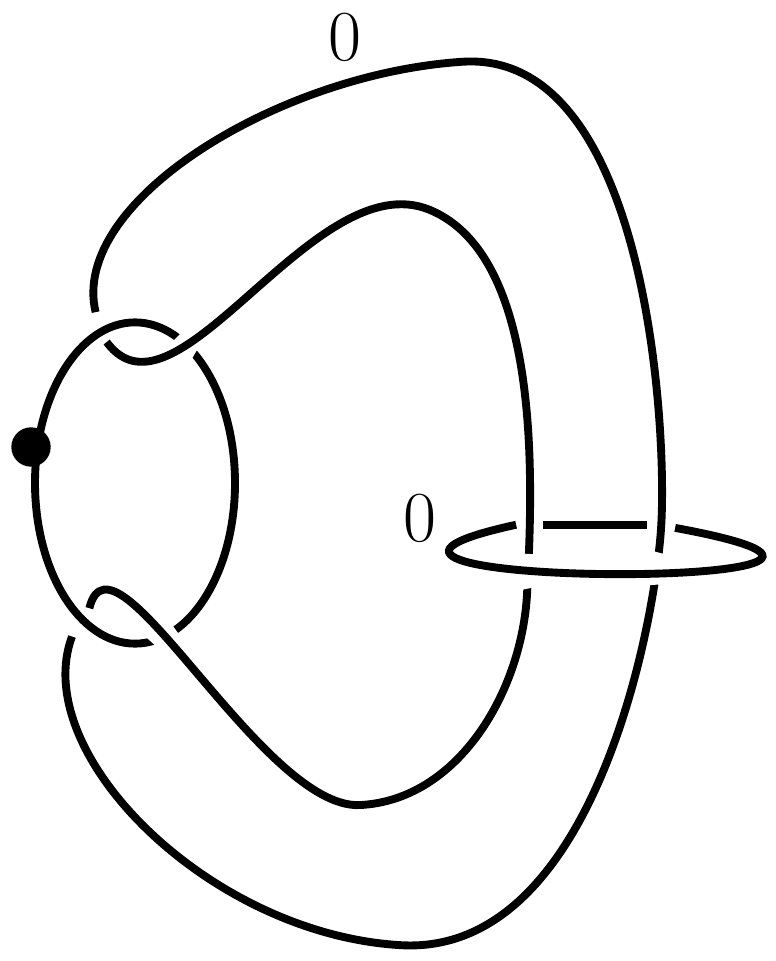}
         \caption{}
         \label{}
     \end{subfigure}
        \caption{A trick due to Moishezon \cite{[Moishezon]} as explained by Gompf in \cite[Proof of Lemma 3]{[Gompf]}}
        \label{fig:Argument}
\end{figure}

The 2-torus $T^2\times\{p\}$, equipped with the framing induced by the product structure in $T^2\times \Sigma_g$, is disjoint from the family of 2-tori (\ref{SetToriA}) and (\ref{SetToriB}), and it defines a framed 2-torus embedded in $\widehat{B}_g$ and one in $B_g^*$. We denote both of them by $T$. 

We now identify the diffeomorphism type of (\ref{Auxiliary Manifold 2}). 

\begin{lemma}\label{Lemma Diffeomorphism 3}Let $\widehat{B}_g$ be the 4-manifold in (\ref{Auxiliary Manifold 1}). There is a diffeomorphism\begin{equation}\label{Diffeomorphism 1'}\widehat{B}_g\approx T^2\times S^2.\end{equation}
 Moreover, through this diffeomorphism the framed 2-torus $T\subset\widehat{B}_g$ is mapped to the framed 2-torus $T^2\times\{p\}\subset T^2\times S^2$ equipped with its standard product framing.
 
Let $B^\ast_g$ be the 4-manifold defined in (\ref{Auxiliary Manifold 2}). There is a diffeomorphism\begin{equation}\label{Diffeomorphism 2'}B^\ast_g\approx (T^2\times S^2)\# 2g(S^2\times S^2).\end{equation}
Moreover, through this diffeomorphism the framed 2-torus $T\subset B^*_g$ is mapped to the canonical 2-torus $T^2\times\{p\}\subset (T^2\times S^2)\# 2g (S^2\times S^2)$ equipped with its product framing.
\end{lemma}
\begin{proof} We first prove the existence of the diffeomorphism (\ref{Diffeomorphism 1'}). Write $T^2\times \Sigma_g$ as $S^1\times \partial (D^2\times \Sigma_g)$, where the $S^1$ factor corresponds to the loop $x$. Use this splitting to view each 4-dimensional torus surgery as a (1 + 3)-dimensional surgery. In particular, the resulting manifold is $\widehat{B}_g = S^1\times Y$, where the 3-manifold $Y$ is obtained from $S^1\times \Sigma_g$ by performing 0-Dehn surgeries around the $2g$ loops $\{1\}\times a_i$ and $\{-1\}\times b_i$ for $i=1,\dots,g$. This is equivalent to adding $2g$ 4-dimensional 2-handles with $0$ framing along the same loops in $D^2\times \Sigma_g$ to get a 4-manifold $Z$ bounded by $Y$ as in Figure \ref{fig:D2xSigmag}. From Figure \ref{fig:Computation}, after some handle slides, handle cancellations and isotopies, we conclude that $Z$ is diffeomorphic to $D^2\times S^2$. In particular, $Y$ is diffeomorphic to $S^1\times S^2$ and the existence of the diffeomorphism (\ref{Diffeomorphism 1'}) follows. The torus $T$ in $S^1\times Y$ is $S^1\times l$, where $l$ is the meridian of the black $0$-framed 2-handle in Figure \ref{fig:D2xSigmag}. Notice that the handle slides in Figure \ref{fig:Computation} do not move the loop $l$. Under the diffeomorphism $Y\approx S^1\times S^2$, the loop $l$ corresponds to the loop $S^1\times \{p\}$. Therefore, the 2-torus $T$ is mapped to $S^1\times (S^1\times \{p\})$ and, after checking the framing of $l$, we conclude that the first part of the lemma holds.

To show the existence of the diffeomorphism (\ref{Diffeomorphism 2'}); we use the Moishezon trick, described above, to break down the surgery along the loops $\gamma_i,\gamma_i'$ by first doing $2g$ multiplicity zero log transforms along the tori $x\times a_i$ and $x'\times b_i$, and then $2g$ loop surgeries along some framed loops $c_1,\dots,c_{2g}$ which are nullhomotopic since $\pi_1(B_g^*) \cong \pi_1 (\widehat{B}_g)$. The manifold $B_g^*$ is the result of $2g$ surgeries along the loops $c_i$ in $\widehat{B}_g = T^2\times S^2$ and it is therefore diffeomorphic to either $\widehat B_g\# 2g (S^2\times S^2)$ or $\widehat B_g\# 2g(S^2\tilde\times S^2)$ \cite[Section 5.2]{[GompfStipsicz]}. To check that the diffeomorphism type is the former, it is sufficient to see that the intersection form of $B^*_g$ is isomorphic to the intersection form of $T^2\times\Sigma_g$, which is even. To see this, consider a set of $4g+2$ surfaces representing a basis for $H_2(T^2\times \Sigma_g;\mathbb{Z})$ and call $Q$ the intersection matrix of this basis. Assume, in addition, that these surfaces are away from the loops (\ref{Set of loops gamma}), hence surgery on such loops leaves the surfaces intact. This yields a set of $4g+2$ surfaces in $B^*_g$, whose intersections give the matrix $Q$, which is unimodular. This, together with the fact that $b_2(B^*_g)=4g+2$, implies that these surfaces must be a basis for $H_2(B^*_g; \mathbb{Z})$, on which the intersection matrix is again $Q$.


We now argue that the framed 2-torus $T\subset B_g^*$ is mapped to canonical 2-torus inside $(T^2\times S^2)\# 2g (S^2\times S^2)$. The diffeomorphism (\ref{Diffeomorphism 1'}) maps the framed loops $c_1,\dots, c_{2g} \subset \widehat{B_g}$ to framed loops $\bar c_1\dots \bar c_{2g} \subset T^2\times S^2$. The loops $\bar c_i$ are nullhomotopic and disjoint from $T^2\times \{p\}$, since the loops $c_i$ are already nullhomotopic and disjoint from $T$ inside $\widehat{B}_g$. Moreover, the loops $\bar c_i$ are nullhomotopic also in $T^2\times (S^2\setminus\{p\})$, given that the inclusion map $T^2\times (S^2\setminus\{p\})\to T^2\times S^2$ induces an isomorphism between fundamental groups. Therefore, each simple loop $\bar c_i$ bounds a smooth 2-disk disjoint from $T^2\times \{p\}$. We can write $T^2\times S^2$ as $(T^2\times S^2)\# S^4$ and, by a smooth ambient isotopy of $T^2\times S^2$, we can move one by one all the loops $\bar c_i$ into the $S^4$ factor keeping fixed a neighborhood of the 2-torus $T^2\times \{p\}$.
In this way the diffeomorphism of tuples
\begin{equation*}
    (\widehat{B}_g,T,c_1,\dots,c_{2g})\approx  ((T^2\times S^2)\#S^4,T^2\times\{p\},\bar c_1,\dots,\bar c_{2g})
\end{equation*}
is inducing a diffeomorphism of pairs\begin{equation}\label{eq: following the torus T from W_g*}
    (B_g^*,T)\approx  ((T^2\times S^2)\#(S^4)^*,T^2\times\{p\})
\end{equation}
where $(S^4)^*$ is the 4-manifold obtained as the result of $2g$  surgery operations along the framed loops $\bar c_1,\dots, \bar c_{2g}$ in $S^4$. By the previous paragraphs we know that $(S^4)^*$ is diffeomorphic to the connected sum of $2g$ copies of $S^2\times S^2$, and we can conclude the proof of this lemma by composing the diffeomorphism (\ref{eq: following the torus T from W_g*}) with a last diffeomorphism between $(T^2\times S^2)\#(S^4)^*$ and $(T^2\times S^2)\#2g(S^2\times S^2)$, which is the identity on a neighborhood of the 2-torus $T^2\times\{p\}$.
\end{proof}

 \begin{figure}[h]
     \centering
     \includegraphics[width=0.5\textwidth]{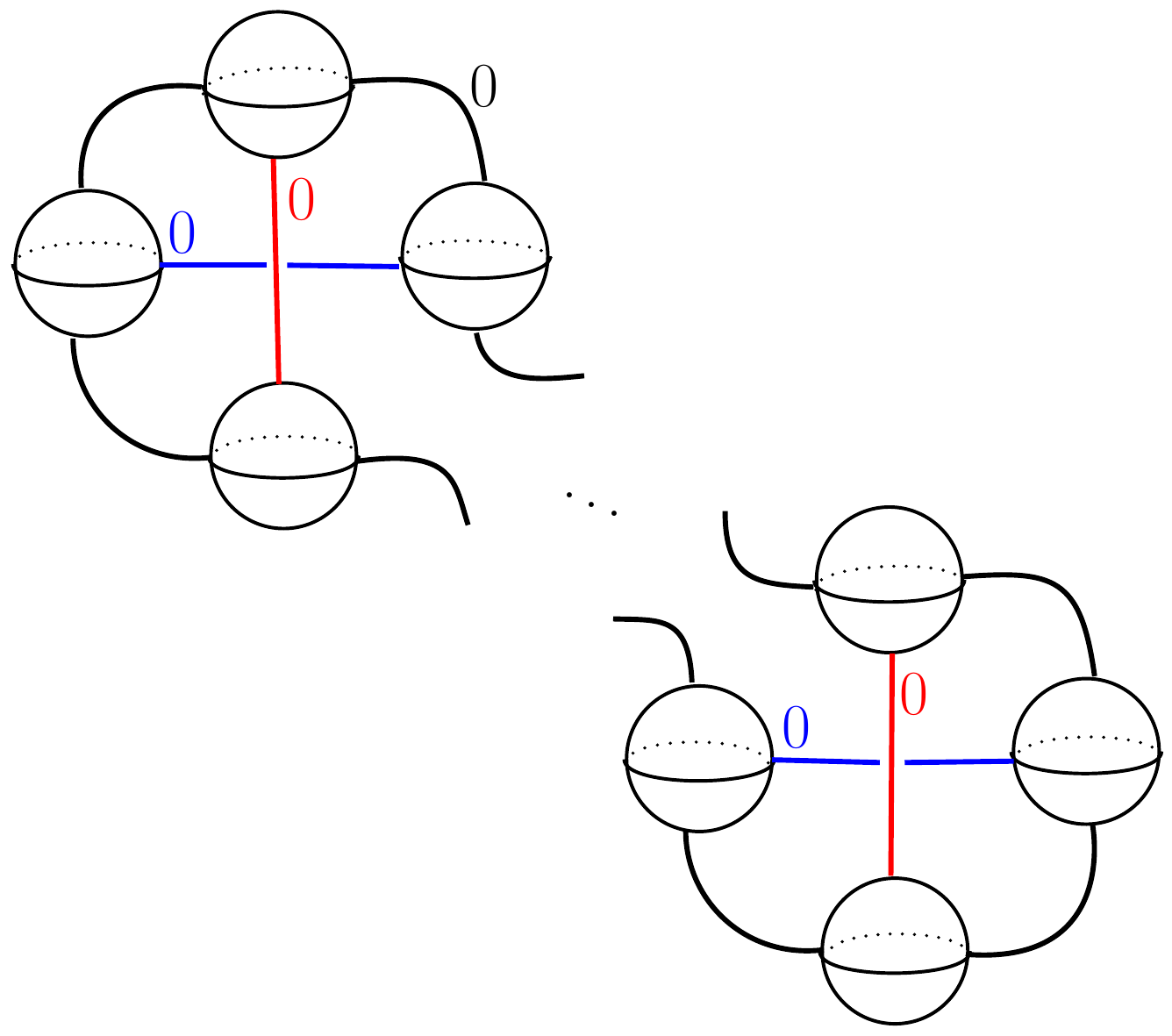}
     \caption{The 1-handles and the black $0$ framed 2-handle represent the disk bundle $D^2\times \Sigma_g$ over the orientable surface of genus $g$. The red loops are the loops $\{1\}\times a_i$ and the blue ones are the loops $\{-1\}\times b_i$. The manifold $Z$ is obtained from  $D^2\times \Sigma_g$ by adding the $0$ framed blue and red 2-handles. The boundary of the 4-manifold $Z$ is  the 3-manifold $Y$.}
     \label{fig:D2xSigmag}
 \end{figure}
 \begin{figure}[h]
     \centering
     \begin{subfigure}[b]{0.3\textwidth}
         \centering
         \includegraphics[width=\textwidth]{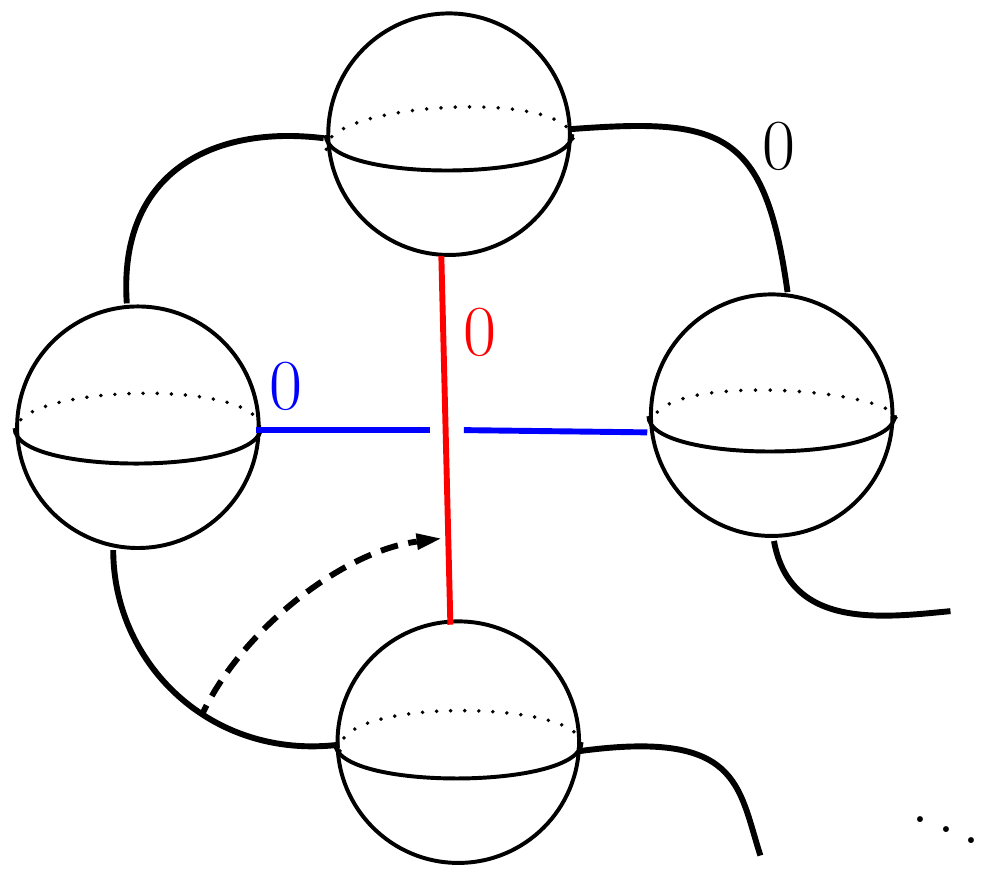}
         \caption{}
         \label{}
     \end{subfigure}
     \hfill
     \begin{subfigure}[b]{0.3\textwidth}
         \centering
         \includegraphics[width=\textwidth]{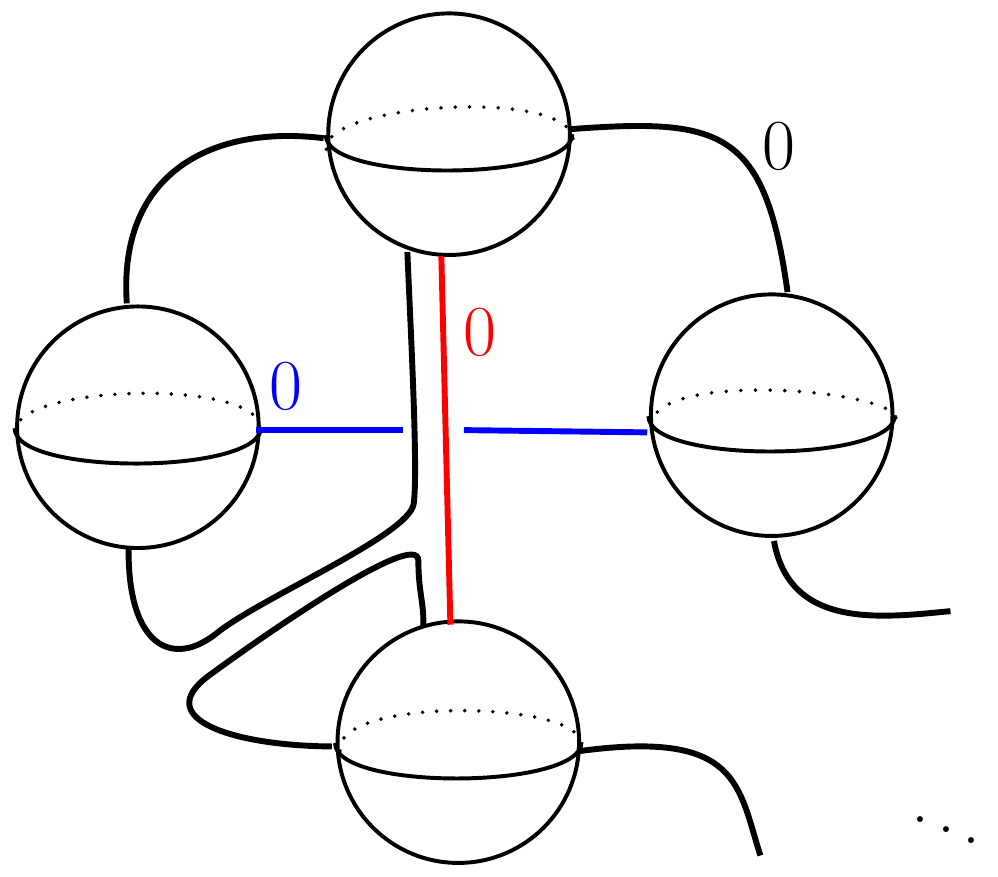}
         \caption{}
         \label{}
     \end{subfigure}
     \hfill
     \begin{subfigure}[b]{0.3\textwidth}
         \centering
         \includegraphics[width=\textwidth]{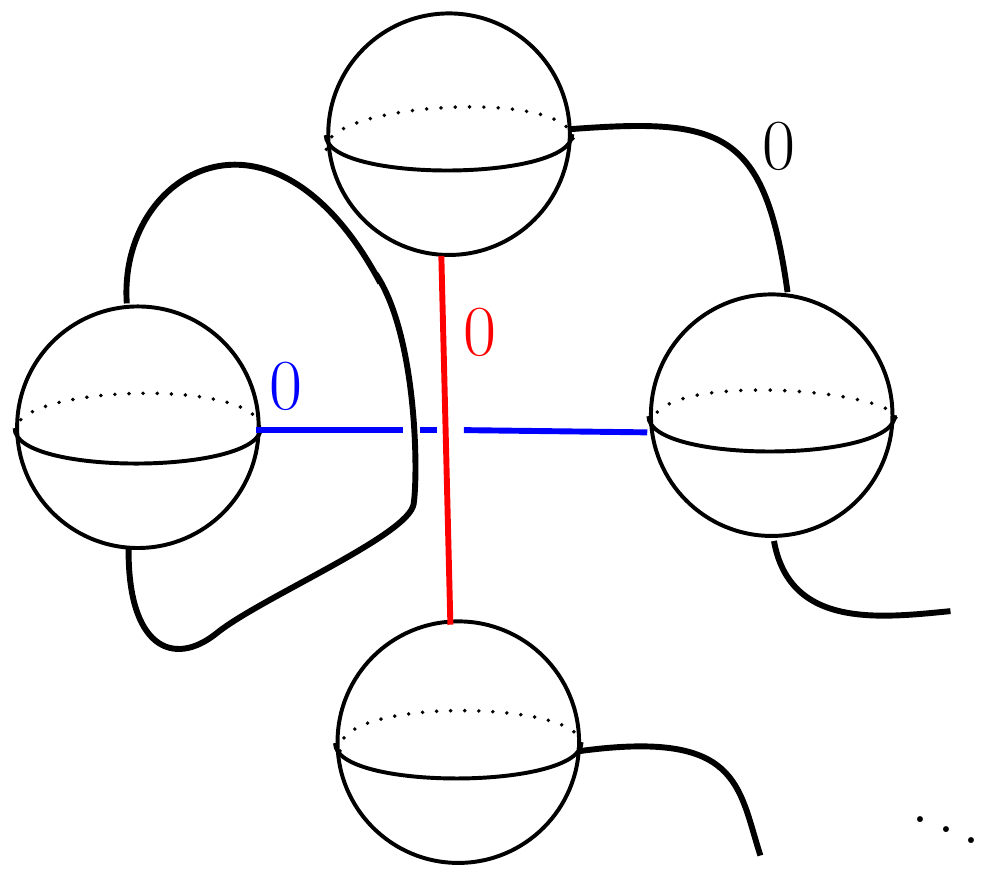}
         \caption{}
         \label{}
     \end{subfigure}
      \hfill
     \begin{subfigure}[b]{0.3\textwidth}
         \centering
         \includegraphics[width=\textwidth]{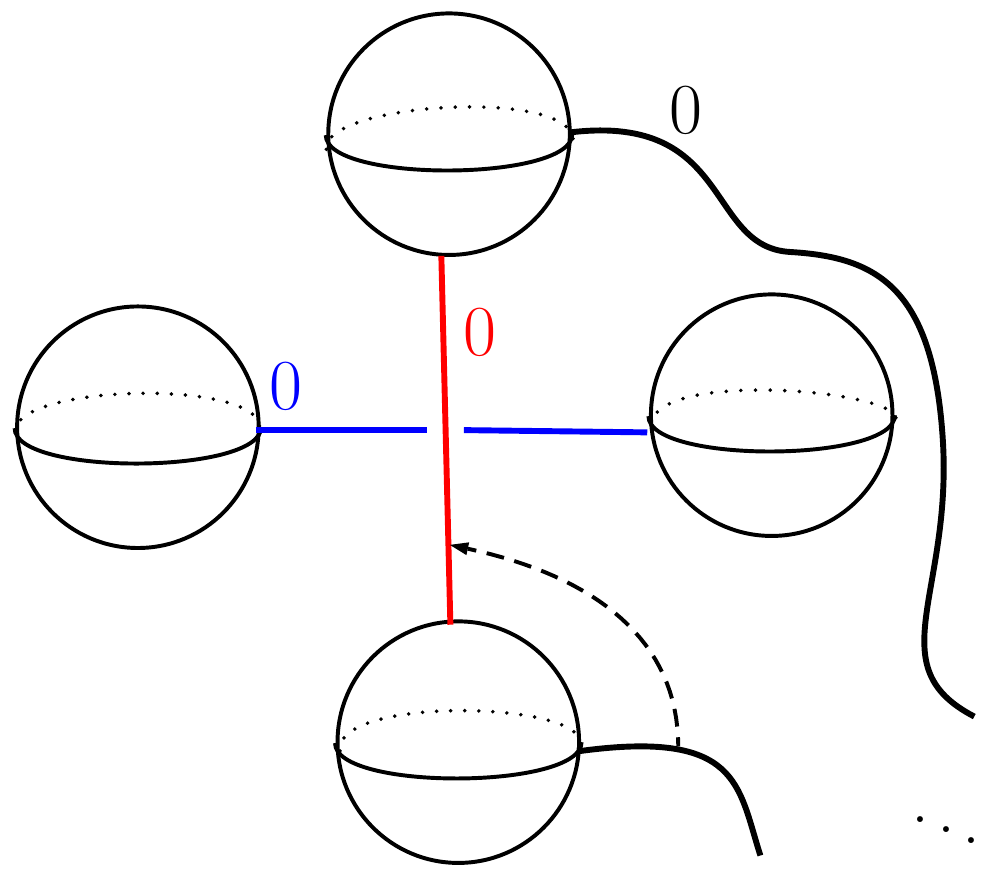}
         \caption{}
         \label{}
     \end{subfigure}
     \hfill
     \begin{subfigure}[b]{0.3\textwidth}
         \centering
         \includegraphics[width=\textwidth]{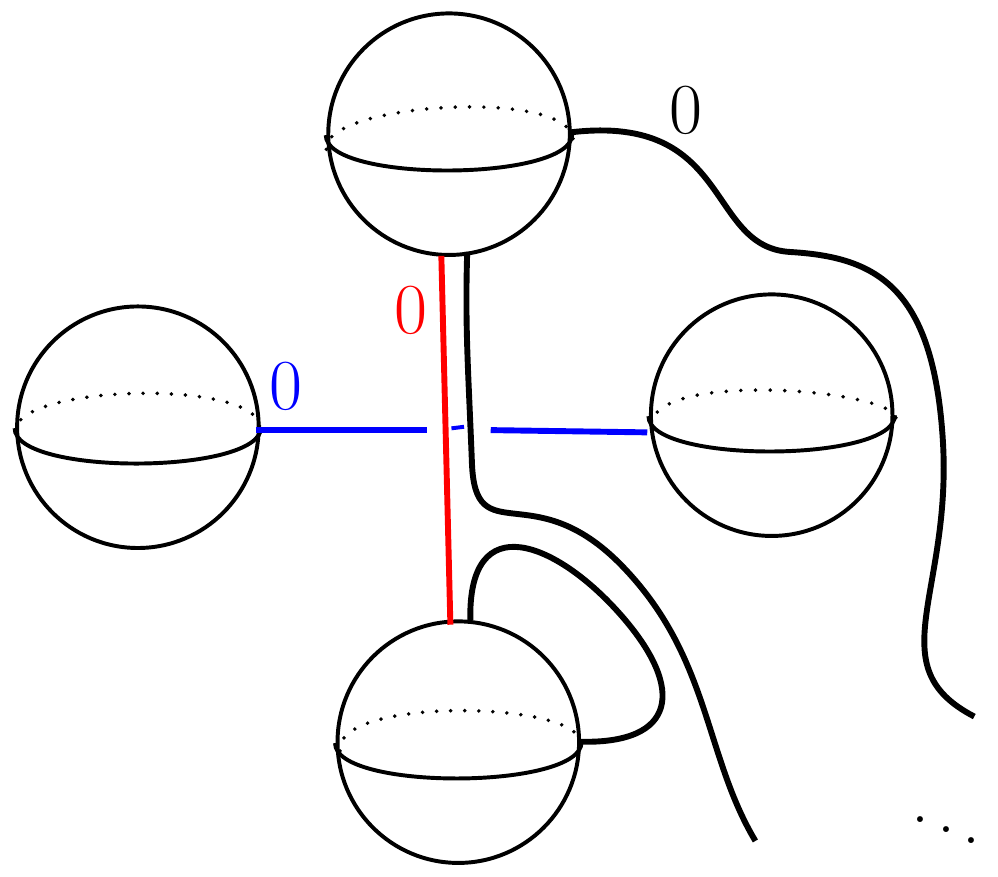}
         \caption{}
         \label{}
     \end{subfigure}
       \hfill
     \begin{subfigure}[b]{0.3\textwidth}
         \centering
         \includegraphics[width=\textwidth]{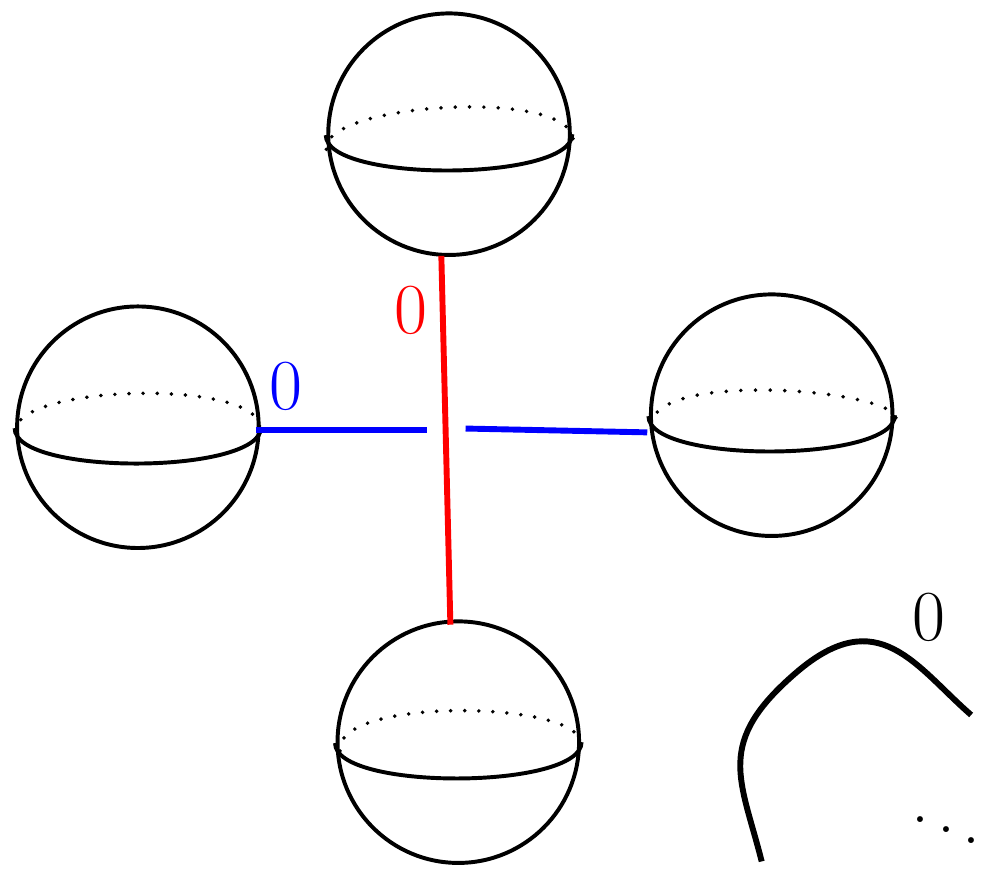}
         \caption{}
         \label{}
     \end{subfigure}

        \caption{A handlebody of the manifold $Z$ in the proof of Lemma 4. The dotted arrows indicate the 2-handle slide performed. If we apply $g$ times this sequence of moves to the handlebody depicted in Figure \ref{fig:D2xSigmag}, we see that the manifold $Z$ has $2g$ cancelling pairs of  1- and 2-handles, along with  a $0$ framed unknotted $2$-handle. Thus, $Z$ is diffeomorphic to $D^2\times S^2$. }
        \label{fig:Computation}
\end{figure}


\subsection{$(T^2\times S^2)\#(S^2\times S^2)$ as a result of surgery to the Kodaira-Thurston manifold}\label{Section Kodaira Thurston} We now describe the Kodaira-Thurston manifold $N$, an essential building block for our constructions in the case $G=F_g$. It is the total space of a $T^2$-bundle over $T^2$ \cite[Section 2.4]{[BaldridgeKirk]}, \cite[Section 2]{[BaldridgeKirk2]} which can also be seen as a product $N = S^1\times Y_{D_a}$  between the $1$-sphere $S^1$ and the mapping torus $Y_{D_a}$ of a Dehn twist $D_a:T^2\rightarrow T^2$ around the loop $a:= S^1 \times \{ p \} \subset T^2=S^1 \times S^1$. For convenience, we also define the auxiliary loop $b:=\{ p' \} \times S^1\subset T^2$. In particular, the $3$-manifold $Y_{D_a}$ is the total space of a $T^2$-bundle over a circle $y$ and its first homology group is $H_1(Y_{D_a}; \mathbb{Z}) = \Z b\oplus \Z y$. Notice that the loop $a$ is homologically trivial in $Y_{D_a}$. Moreover, we write the $T^2$-bundle over $T^2$ structure of $N =  S^1\times Y_{D_a} = x\times Y_{D_a}$ as\begin{equation}\label{KT Loops}a\times b\hookrightarrow N\rightarrow x\times y\end{equation} where our notation emphasizes the loops contained in the 2-torus fiber and the 2-torus section that generate the first homology group $H_1(N; \mathbb{Z}) = \Z x\oplus \Z y\oplus \Z b$. Note that the fiber 2-torus $a\times b$ is geometrically dual to the 2-torus section $T:= x\times y$. Moreover, there is another pair of geometrically dual 2-tori inside $N$, that we denote by $x\times b$ and $y\times a$ respectively, which are disjoint from both the fiber and the section; see \cite[Section 2]{[BaldridgeKirk]} for a detailed description of these submanifolds. 

We equip $N$ with a symplectic structure as in for which the 2-torus $x\times b\subset N$ is Lagrangian and use the unique Lagrangian framing when we perform surgery along this submanifold; this is explained in detail in \cite[Section 2.1]{[BaldridgeKirk]}, \cite{[HoLi]}. Let $m$ and $l$ be the Lagrangian pushoffs of $x$ and $b$, respectively, and let $\mu$ be the meridian of $x\times b$, i.e., a curve isotopic to $\{p\}\times \partial D^2\subset \partial \nu(x\times b)$. The triple $\{m, l, \mu\}$ is a basis for the group $H_1(\partial \nu(x\times b); \mathbb{Z})$; see \cite[Section 2.1]{[BaldridgeKirk]} for further details. We now define the $4$-manifold\begin{equation}\label{Auxiliary Manifold}\widehat{N} = (N\setminus \nu(x\times b)) \cup_{\varphi_0} (T^2\times D^2),\end{equation} by using a diffeomorphism\begin{equation}\label{Choice Diffeo}\varphi_0: \partial \nu(x\times b)\rightarrow T^2\times \partial D^2\end{equation} satisfying the condition $\varphi_0(l)_\ast = [\{p\}\times \partial D^2]\in H_1(T^2 \times \partial D^2; \Z)$. In other words, $\widehat N$ is obtained from $N$ by applying a multiplicity zero log transform.

We identify the diffeomorphism type of $\widehat N$ and of another useful $4$-manifold obtained by the Kodaira-Thurston manifold via loop surgery in the following lemma.

\begin{lemma}\label{Lemma Diffeomorphism 1}Let $\widehat{N}$ be the 4-manifold in (\ref{Auxiliary Manifold}) 
There is a diffeomorphism\begin{equation}\label{Diffeomorphism 1}\hat{N}\approx T^2\times S^2.\end{equation}
Let $N^\ast$ be the 4-manifold that is obtained from the Kodaira-Thurston manifold by performing surgery along the based loop $b$ with respect to the framing induced by the Lagrangian framing on $x\times b\subset N$. There is a diffeomorphism\begin{equation}\label{Diffeomorphism 2}N^\ast\approx (T^2\times S^2)\# (S^2\times S^2).\end{equation}
Moreover, the framed 2-torus $T\subset N$ is disjoint from the loop $b$, so it defines a framed 2-torus  $T\subset N^*$ which is mapped through the diffeomorphism (\ref{Diffeomorphism 2}) to the canonical 2-torus $T^2\times\{p\}\subset (T^2\times S^2)\# (S^2\times S^2)$ equipped with its product framing.

\end{lemma}

\begin{figure}
  \includegraphics[width= 0.8\textwidth]{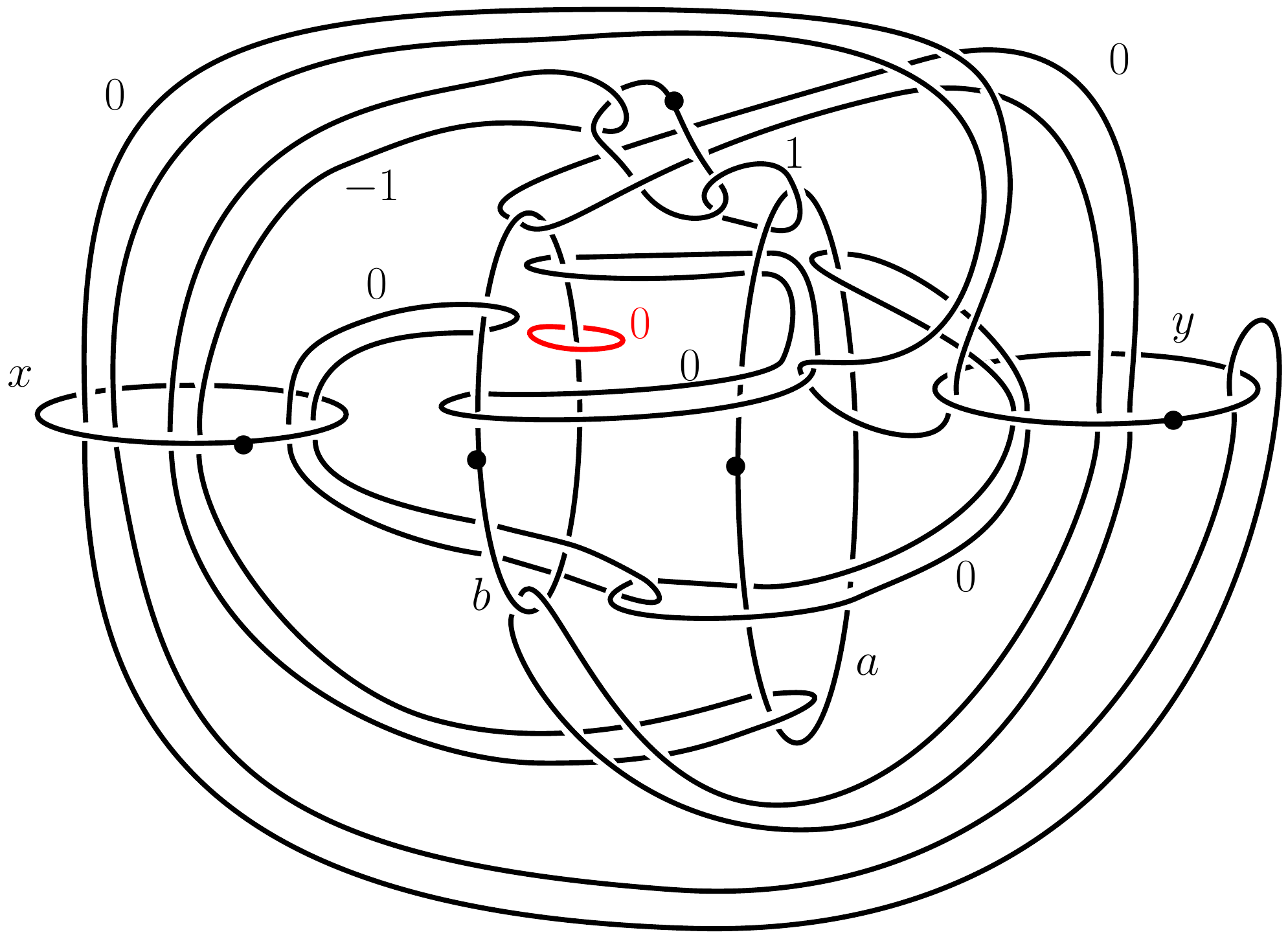} 
  \caption{Handlebody depiction of the Kodaira-Thurston manifold $N$ and the 4-manifold $N^\ast$.}
 \label{fig: handlebody N*}
\end{figure}

We present two proofs of the lemma for the convenience of the reader.


\begin{proof} Given the relevance of the 4-manifold $N^\ast$ in the proof of Theorem \ref{Theorem A}, we begin with a proof of the existence of the diffeomorphism (\ref{Diffeomorphism 2}). The handlebody of the 4-manifold $N^\ast$ is given in Figure \ref{fig: handlebody N*}, where its three 3-handles and its 4-handle are not drawn. The existence of the diffeomorphism (\ref{Diffeomorphism 2}) is seen by first using the 0-framed circle that links once the second dotted circle from left to right in Figure \ref{fig: handlebody N*} to unlink all other attaching spheres of the 2-handles from this dotted circle. Cancel this 1-handle and 2-handle pair. Straight-forward handle slides unlink the diagram even further, and two more 1- and 2-handle cancellations yield a handlebody of $(T^2\times S^2)\#(S^2\times S^2)$.

In order to draw this handlebody of $N^\ast$, we build heavily on work of Akbulut to first draw a Kirby diagram of the Kodaira-Thurston manifold. Equip the 4-torus $T^4 = x\times y\times a\times b$ with the product symplectic form. The Kodaira-Thurston manifold $N$ is obtained from the 4-torus $T^4$ by applying one Luttinger surgery to the Lagrangian 2-torus $x\times a$ along the Lagrangian pushoff of $a$ \cite{[AurouxDonaldsonKatzarkov]}. We use this description of $N$ in order to draw its handlebody by building on the handlebody of the 4-torus drawn in \cite[Figure 4.5]{[Akbulut2]} and the depiction of Luttinger surgery \cite[Section 6.4]{[Akbulut2]}. The handlebody of the Kodaira-Thurston manifold is given in Figure \ref{fig: handlebody N*} without the 0-framed 2-handle that links the second dotted circle from left to right along with four 3-handles and a 4-handle; cf. \cite[Figure 14.35]{[Akbulut2]}. 

The second argument to prove the lemma is as follows. The existence of the diffeomorphism (\ref{Diffeomorphism 1}) follows from a standard argument; see Baldridge-Kirk in \cite[Proof Lemma 2]{[BaldridgeKirk2]}. The 4-manifold $\widehat{N}$ is diffeomorphic to the product $S^1\times M^3$ of the circle with a 3-manifold $M^3$ that is obtained from the 3-torus $T^3 = y\times a\times b$ by applying a $1$-Dehn surgery along $a$ and a $0$-Dehn surgery along $b$. The resulting 3-manifold $M^3$ has infinite cyclic fundamental group. By a similar Kirby diagram argument as for the diffeomorphism (\ref{Diffeomorphism 1'}), we see that $M^3$ is diffeomorphic to $S^1\times S^2$ and we conclude that $\widehat{N} = S^1\times (S^1\times S^2) = T^2\times S^2$.  

We now prove the existence of the diffeomorphism (\ref{Diffeomorphism 2}) using the Moishezon trick;
it implies that the 4-manifold $N^\ast$ is obtained by doing surgery along a nullhomotopic loop in $\widehat{N}$. Therefore, $N^*$ is diffeomorphic to either $\widehat N \# (S^2\times S^2)$ or $\widehat N \# (S^2\tilde\times S^2)$. To check that the diffeomorphism type is the former, it is sufficient to proceed as in Lemma \ref{Lemma Diffeomorphism 3}. Using (\ref{Diffeomorphism 1}), we conclude that $N^\ast \approx (T^2\times S^2)\#(S^2\times S^2)$. Moreover, this diffeomorphism can be chosen to send the framed 2-torus $T$ to the canonical 2-torus with its product framing (cf. proof of Lemma \ref{Lemma Diffeomorphism 3}).
\end{proof}

\subsection{The 4-manifolds $N_g$ for $g\in \N$}\label{Section Manifold N_g}We now generalize the construction of the Kodaira-Thurston manifold $N$ in the previous section to produce a 4-manifold $N_g$ that will be the building block $B_{F_g}$ mentioned in step (b) of the strategy in Section \ref{Strategies} for any natural number $g$. The symplectic $4$-manifold $N_g$ is defined as\begin{equation}\label{Definition N_g}
    N_g := S^1 \times Y_g
\end{equation}where $Y_g$ is the mapping torus of\begin{equation*}
    \varphi_g := D_{a_1} \circ ... \circ D_{a_g} : \Sigma_g \rightarrow \Sigma_g 
\end{equation*}and $D_{a_i}$ is a Dehn twist around the loop $a_i \subset \Sigma_g$ defined in Figure \ref{fig:T2xSigmag}.




\begin{remark}\label{Remark Construction}It is possible to describe (\ref{Definition N_g}) as a symplectic sum\begin{equation}\label{kodaira thurston fibersum}
N_g = N_1\#_{T^2} \cdots \#_{T^2}N_1
\end{equation}of g copies of the Kodaira-Thurston manifold $N_1 = N$. As we already saw in Section \ref{Section Kodaira Thurston}, $N$ is obtained by taking the product of $S^1$ with a mapping torus of a Dehn twist of $T^2$. This is generalized in the construction (\ref{Definition N_g}) of $N_g$ given that a surface of genus $g$ can itself be deconstructed as a connected sum\begin{equation*} \Sigma_g = \underbrace{T^2\#\cdots \#T^2}_{g \text{ times}}\end{equation*} of $g$ copies of the 2-torus.

\end{remark}

The 4-manifold $N_g$ contains a symplectic 2-torus\begin{equation}\label{T}T = x \times y \subset N_1 \setminus \nu (T^2) \subset N_g\end{equation} as well as a Lagrangian 2-torus $x \times b$ for each copy of $N_1$ in (\ref{kodaira thurston fibersum}). These 2-tori can also be seen as $x'\times b_i$ where $x'$ is a parallel copy of $x$ and $b_i$ is the inclusion of $b$ in $N_g$ through the $i^{th}$ copy of $N_1$ in (\ref{kodaira thurston fibersum}). Let $\gamma_i '$ be the framed loops 
\begin{equation}\label{Loops Lemma 3}\gamma_i'=\{p'\}\times b_i\subset x'\times b_i\subset N_g\end{equation}
with framing induced by the Lagrangian framing on the 2-torus $x' \times b$. We gather together the framed loops (\ref{Loops Lemma 3}) into a 1-dimensional submanifold\begin{equation}\label{loops2}
\mathcal{L}_{F_g}=\gamma_1' \sqcup \dots \sqcup \gamma_g'.
\end{equation}

The following homotopical properties of $N_g$ are useful for our purposes.
 
\begin{lemma}\label{Lemma Properties Manifold} Let $g\in\N$.

$\bullet$ The fundamental group of $N_g$ is generated by the inclusion of the fundamental group of the torus $j_*\pi_1(T)$ and the loops $\tilde\gamma_i$ obtained by connecting the loops $\gamma_i'$ to a base point.

$\bullet$ There is an isomorphism 
    \begin{equation}\label{eq: QuotientIsFree}
        \frac{\pi_1(N_g)}{\langle j_*\pi_1(T)\rangle_N }\approx F_g
    \end{equation} where $\langle j_*\pi_1(T)\rangle_N$ is the normal subgroup generated by $j_*\pi_1(T)$ and the loops $\tilde\gamma_i$ correspond to generators.
\end{lemma}
\begin{proof}

We know that $\pi_1(Y_g)$ is given by an HNN extension, which yields 
\begin{equation*}
    \pi_1 (N_g) \cong \langle x \rangle  \oplus \langle y, a_1 , b_1 ,...,a_g,b_g \mid  [y,a_i]= [y,b_i]a^{-1}= \prod_{j=1}^g [a_j,b_j] =1 \text{ for } i=1, \cdots, g\rangle
\end{equation*}
with $j_* (\pi_1(T)) = \langle x, y \mid [x,y] \rangle $. We conclude that
\begin{equation*}
    \frac{\pi_1(N_g)}{\langle j_*(\pi_1(T))\rangle _N} \cong \langle b_1 , b_2 ,...,b_g \rangle \cong F_g,
\end{equation*}where $x\in \pi_1(T^2\times \Sigma_g)$ is the homotopy class of the loop $x\times\{p\}\subset T^2\times \Sigma_g$ connected to a base point, and similarly for $y,a_i$ and $b_i$ for $i=1,\dots,g$.

\end{proof}

Perform $g$ surgeries to $N_g$ along the framed loops (\ref{Loops Lemma 3}) to produce a closed 4-manifold\begin{equation}\label{Definition N_g*}
    N_g^\ast = N_g\setminus (\bigsqcup_{i=1}^g \nu(\gamma_i')\cup \bigsqcup_{i=1}^{g} (D^2\times S^2))\end{equation}with fundamental group $\pi_1(N^*_g)=\Z\oplus \Z$. The  2-torus $T$ is disjoint from the seam of the surgeries in (\ref{Definition N_g*}) and it defines an embedded framed 2-torus $T$ in $N_g^*$.

\vspace{1mm}
We now identify the diffeomorphism type of the 4-manifold $N_g^*$.

\begin{lemma}\label{Lemma Diffeomorphism 2} Let $g\in\N$ and let $N_g^*$ be the manifold defined in (\ref{Definition N_g*}). There is a diffeomorphism\begin{equation}\label{Diffeomorphism Ng*}N_g^*\approx (T^2\times S^2)\# g (S^2\times S^2).\end{equation}
Moreover, through this diffeomorphism the framed 2-torus $T\subset N_g^*$ is mapped to the canonical 2-torus $T^2\times\{p\}\subset (T^2\times S^2)\# g(S^2\times S^2)$ equipped with its product framing.
\end{lemma}

\begin{proof}A proof of Lemma \ref{Lemma Diffeomorphism 2} by induction on $g$ is obtained using Lemma \ref{Lemma Diffeomorphism 1} and the construction of $N_g$ in terms of a generalized fiber sum of g copies of the Kodaira-Thurston manifold $N_1$ as it is described in Remark \ref{Remark Construction}. The details are left to the reader.
\end{proof}

\subsection{Infinitely many inequivalent smooth structures}\label{Section Smooth Structures} 

The smooth 4-manifolds of the fourth clause of Theorem \ref{Theorem A} are introduced in the following proposition.

\begin{proposition}\label{Proposition Smooth Structures 1}Let $M$ be an admissible 4-manifold in the sense of Definition \ref{Definition Admissible Manifold}. 
Let $B_G$ be the building block in step (b) of Section \ref{Strategies} that was defined in Section \ref{Section manifold W_g} and Section \ref{Section Manifold N_g} for the groups $G = \pi_1(\Sigma_g)$ and $G=F_g$, respectively.

The generalized fiber sum
\begin{equation}\label{example1} M\#_{T^2} B_G = (M\setminus \nu(T_2))\cup (B_G \setminus \nu(T))\end{equation} 
admits infinitely many pairwise inequivalent smooth structures $\{Z_K \mid K \in \mathcal{K} \}$.

Moreover, the fundamental group of (\ref{example1}) is isomorphic to $G$ and it is generated by the homotopy classes of the loops obtained by connecting every component of $\mathcal{L}_G$ (\ref{Set of loops gamma}) (\ref{loops2}) to a common base point.
\end{proposition}


\begin{proof} For a knot $K\in \mathcal{K}$, perform Fintushel-Stern knot surgery to M along $T_1$ to obtain the $4$-manifold $M_K$. Notice that we still have an inclusion\begin{equation*}T_2\subset (M \setminus \nu(T_1))\xhookrightarrow{i} M_K,\end{equation*} since this cut and paste operation can be performed away from the $2$-torus $T_2$. We abuse notation and denote $i(T_2)$ by $T_2$, and we make sure that the context resolves any confusion. The intersection forms of $M$ and $M_K$ are isomorphic as computed in \cite[V.4.1]{[Hamilton]}\cite{[FintushelStern3]}. Moreover, there is an isomorphism of the second homology groups sending the homology class $[T_2]\in H_2(M; \mathbb{Z})$ onto $[T_2] \in H_2(M_K; \mathbb{Z})$ \cite[Chapter V]{[Hamilton]}. By Freedman's theorem \cite{[Freedman]} this isomorphism can be realised by a homeomorphism $f:M\rightarrow M_K$ such that the induced map in homology $f_*$ sends the homology class $[T_2]$ onto itself
\[\begin{tikzcd}  [row sep=small]
 H_2(M;\mathbb{Z})\arrow[r,"f_*"] &  H_2(M_K; \mathbb{Z})\\
{[T_2]}\arrow[r]\arrow[u,phantom, sloped, "\in"] & {[T_2]}\arrow[u, phantom, sloped, "\in"].
\end{tikzcd}\]
    By our hypothesis, the $2$-tori $f(T_2), T_2 \subset M_K$ have simply connected complement. A result of Sunukjian \cite[Theorem 7.1]{[Sunukjian2]} implies the existence of a homeomorphism of pairs\begin{equation}g:(M_K,f(T_2))\rightarrow (M_K, T_2)\end{equation}and we obtain the desired homeomorphism of pairs \begin{equation}\label{Needed}g\circ f: (M,T_2) \rightarrow (M_K,T_2).\end{equation}
    We proceed to fiber sum both $M$ and $M_K$ with $B_G$. In particular, we perform the generalized fiber sum by identifying the $2$-tori $T_2$ in $M$ and $M_K$ with the framed 2-torus $T \subset B_G$. In order to do so, the 2-torus $T_2\subset M_K$ is endowed with a framing $\tau_K$ that is the restriction of the map  $\tau\circ(g\circ f)^{-1}$ to $\overline{\nu(T_2)}$. We fix any framing for the 2-torus $T_2$ in $M$. Notice that $\tau_K$ is indeed a smooth framing, since we can assume that the map $g\circ f$ is a smooth diffeomorphism around $T_2$. We define the 4-manifolds\begin{equation}\label{Manifold ZZk}
    Z:=M\#_{T^2}B_G\qquad Z_K:=M_K\#_{T^2}B_G.\
    \end{equation}Given our chosen framings, the homeomorphism (\ref{Needed}) extends via the identity map to a homeomorphism $F:Z\to Z_K$. The claim regarding the fundamental group $\pi_1(M \#_{T_2} B_G) = G$ follows from Lemma \ref{Lemma Properties Manifold} and the Seifert-Van Kampen theorem.


To show that the set $\{Z_K \mid K \subset \mathcal{K} \}$ consists of pairwise non-diffeomorphic 4-manifolds, we look at the Seiberg-Witten invariant of each of its elements as a Laurent polynomial in the group ring $\Z[H_2(Z_K; \Z)]$ for every $K\in \mathcal{K}$ \cite[Lecture 2]{[FintushelStern2]}. Recall that the Seiberg-Witten invariants $\{SW_X(k)\mid k\in H_2(X)\}$ of a 4-manifold $X$ can be combined into an element of $\Z[H_2(X; \Z)]$ by associating a formal variable $t_{\alpha}$ to each homology class $\alpha \in H_2(X; \Z)$ and by setting \begin{equation}\label{SW Invariant}\mathcal{SW}_X = \sum SW_X(k)\cdot t_k\end{equation}where the sum is taken over all the basic classes $k\in H_2(X;\Z)$; see  \cite[p. 200]{[Fintushel]}, \cite{[Sunukjian1]} for further details. An admissible 4-manifold contains at least one basic class and hence the invariant (\ref{SW Invariant}) is nonzero, and the same holds for the symplectic 4-manifold $B_G$ by a result of Taubes \cite{[Taubes1]}.


We proceed to set up the use of a gluing formula due to Taubes \cite{[Taubes2]} to compute (\ref{SW Invariant}) for $Z_K$; see \cite[Section 13.9]{[Akbulut2]}. The boundary of $B_G \setminus \nu(T)$ is diffeomorphic to $S^1\times S^1\times \partial D^2$ by the diffeomorphism coming from the framing of $T\subset B_G$ and the rim 2-tori\begin{center}$S^1\times \{p\} \times \partial D^2$ and $\{p\} \times S^1 \times \partial D^2$\end{center} bound 3-manifolds in $B_G\setminus \nu(T)$. This fact combined with the Mayer-Vietoris sequence for $Z_K = (M_K \setminus \nu(T_2) )\cup(B_G \setminus \nu(T)) $ implies that the inclusion induced homomorphism\begin{equation*}j_*:H_2(B_G \setminus \nu (T); \mathbb{Z}) \rightarrow H_2(Z_K; \mathbb{Z})\end{equation*} is injective. The Mayer-Vietoris sequence for $M_K = (M_K\setminus \nu(T_2)) \cup \nu(T_2)$ yields that the inclusion induced homomorphisms\begin{equation*}i_* :H_2(M_K \setminus \nu (T_2); \mathbb{Z}) \rightarrow H_2(Z_K; \mathbb{Z})\end{equation*}and\begin{equation*}i_*' :H_2(M_K \setminus \nu (T_2); \mathbb{Z}) \rightarrow H_2(M_K; \mathbb{Z})\end{equation*} have the same kernel\begin{equation*}
    ker (i_*')= ker (i_*) =\langle [\tau_K^{-1}(\{p\}\times S^1 \times \partial D^2)] ,[\tau_K^{-1}(S^1 \times \{p\} \times \partial D^2)] \rangle.\end{equation*}

On the other hand, $\mathcal{SW}_{M_K} = \mathcal{SW}_M \cdot \Delta_K(e^{2[T_1]}) \neq 0$ by \cite{[FintushelStern3]} (see \cite[p. 201]{[Fintushel]}, \cite[Proposition 13.21]{[Akbulut2]}). Taubes' gluing formula \cite[Theorem 1.1]{[Taubes2]} implies that 
\begin{equation*}
    \mathcal{SW}_{M_K}=i_*'(\mathcal{SW}_{M_K \setminus \nu (T_2)} ) \cdot i_*''(\mathcal{SW}_{\nu(T_2)}).
\end{equation*}
It follows that $\mathcal{SW}_{M_K \setminus \nu (T^2)} \not \in ker(i_*')=ker(i_*)$.


Moreover, $c:=i_* ([T_1']) \neq 0$ where $T_1 '$ is a pushoff of $T_1$ inside $M_K \setminus \nu (T_2)$. We apply again Taubes' gluing formula \cite[Theorem 1.1]{[Taubes2]} and get
\begin{align*}
    \mathcal{SW}_{Z_K} &= i_* (\mathcal{SW}_{M_K\setminus \nu (T_2)} )\cdot j_* (\mathcal{SW}_{B_G\setminus \nu (T)}) \\
    &= i_* ( \Delta_K (e^{2[T_1']}) ) \cdot i_*(\mathcal{SW}_{M \setminus \nu (T_2)})  \cdot j_* (\mathcal{SW}_{(B_G\setminus \nu (T)})  \\ 
    &=   \Delta_K (e^{2c})  \cdot \mathcal{SW}_Z \\
    &\ncong \mathcal{SW}_Z
\end{align*}
for a knot $K$ with non-trivial Alexander polynomial. The $\mathcal{SW}$-invariants, thus, are different\cite{[Sunukjian1]}, and we conclude that $Z$ and $Z_K$ are non-diffeomorphic. Moreover, for two knots $K_1$ and $K_2$ with Alexander polynomials $\Delta_{K_1}\neq \Delta_{K_2}$, we have that the 4-manifolds $Z_{K_1}$ and $Z_{K_2}$ have different $\mathcal{SW}$-invariants and, therefore, they are non-diffeomorphic. This concludes the proof of the proposition.

\end{proof}

It is possible to give a more explicit computation of (\ref{SW Invariant}) of $Z_K$ in the proof of Proposition \ref{Proposition Smooth Structures 1} by using other gluing formulas \cite[Section 1]{[Fintushel]}. B. D. Park computed \begin{center}$\mathcal{SW}_{T^2\times \Sigma_g} = (t^{-1} - t)^{2g - 2}$ and $\mathcal{SW}_{T^2\times {\Sigma_g^0}} = (t^{-1} - t)^{2g - 1}$\end{center} for $t = [T^2\times \{p\}]$ and $\Sigma_g^0 = \Sigma_g\setminus D^2$ in \cite[Corollary 19]{[BDPark]}. Since the symplectic 4-manifold $N_g$ is obtained by applying Luttinger surgeries to $T^2\times \Sigma_g$, the invariants $\mathcal{SW}_{N_g}$ and $\mathcal{SW}_{N_g\setminus \nu(T)}$ can be computed using B. D. Park's work.

\begin{remark}\label{irrelevant framing}
Notice that the diffeomorphism type of the generalized fiber sum $M\#_{T_2} (T^2 \times S^2)$ is the same regardless the choice of framing $\tau$ or $\tau_K$ for $T_2$ and the product framing for $T^2 \times \{p\}\subset T^2\times \Sigma_g$. 
\end{remark}

\subsection{The ambient 4-manifolds of Theorem \ref{Theorem A}}\label{Section Diffeomorphism Type}We now identify the diffeomorphism types of the ambient 4-manifolds of Theorem \ref{Theorem A}. Since the codimension three submanifold $\mathcal{L}_G\subset B_G$ is disjoint from $T$, it is embedded in the 4-manifold $Z_K$ constructed in (\ref{Manifold ZZk}) and each of its components is framed. We define $Z_K^*$ to be the 4-manifold obtained by doing loop surgery along each component of $\mathcal{L}_G\subset Z_K$. Moreover, we can define $\Gamma$ to be the 2-link in $Z_K^*$ given by the belt 2-spheres of the surgeries.

\begin{proposition}\label{Proposition Diffeo Ambient}For every knot $K\in \mathcal{K}$, there is a diffeomorphism\begin{equation}\label{Ambient Diffeomorphism 1}Z^\ast_{K} \approx M\#n(S^2\times S^2).\end{equation}

\end{proposition}

\begin{proof} We fix a knot $K \in \mathcal{K}$ and take a closer look at the assemblage of $Z^*_{K}$ to prove the existence of the diffeomorphsm (\ref{Ambient Diffeomorphism 1}). 
Recall that the manifold $Z_K$ is defined as the generalized fiber sum of $M_K$ and $B_G$ along the 2-tori $T_2\subset M_K$ with framing $\tau_K$ and $T\subset B_G$ with the lagrangian framing. Since  $\mathcal{L}_G\subset B_G$ is disjoint from $T$, the 4-manifold $Z_K^*$ is the generalized fiber sum of $M_K$ and $B_G^*$ along $T_2\subset M_K$ and $T\subset B_G^*$. Lemma \ref{Lemma Diffeomorphism 3} and Lemma \ref{Lemma Diffeomorphism 2} say that there is a diffeomorphism\begin{equation*}B_G^*\approx (T^2\times S^2)\# n (S^2\times S^2)\end{equation*}that sends the framed 2-torus $T$ to the canonical 2-torus\begin{equation}\label{Canonical Torus}T^2\times \{p\}\subset (T^2\times S^2)\# n (S^2\times S^2).\end{equation} Therefore, the 4-manifold $Z_K^*$ is diffeomorphic to the generalized fiber sum of $M_K$ and $(T^2\times S^2)\# n(S^2\times S^2)$ along the 2-tori $T_2\subset M_K$ and (\ref{Canonical Torus}). To sum up, using remark \ref{irrelevant framing}, we have\[Z_K^*=M_K\#_{T^2} B_G^*\approx M_K\#_{T^2} (T^2\times S^2)\#n(S^2\times S^2)\approx M_K\#n(S^2\times S^2).\] A result of Akbulut \cite{[Akbulut1]}, Auckly \cite{[Auckly]} or Baykur \cite{[Baykur]} allows us to conclude that $Z^\ast_{K}$ is diffeomorphic to the connected sum $M\#n(S^2\times S^2)$ for any $K\subset S^3$. \end{proof}


\subsection{The 2-links of Theorem \ref{Theorem A} and topological isotopy}\label{Section TOPIsotopy} In this section, we construct the $2$-links of Theorem \ref{Theorem A} and show that they are pairwise topologically isotopic and componentwise topologically unknotted.

\begin{proposition}\label{Proposition Topologically Isotopic}
For $G\in \{ F_g , \pi_1(\Sigma_g)\}$ and $n=rk(G)$, there is an infinite collection of $n$-component 2-links \begin{equation}\label{eq: LINKS GAMMA_K}\{\Gamma_K \mid K\in\mathcal{K}\}\end{equation}smoothly embedded in $M\#n(S^2\times S^2)$ that are pairwise topologically isotopic and componentwise topologically unknotted.

The 4-manifold that is obtained from $M\#n(S^2\times S^2)$ by doing surgeries along every component of (\ref{eq: LINKS GAMMA_K}) is diffeomorphic to the $4$-manifold $Z_K$ defined in (\ref{Manifold ZZk}) for any knot $K \in \mathcal{K}$. 
\end{proposition}

\begin{proof} Let $U\subset S^3$ be the unknot. For any knot $K \in \mathcal{K}$, there is a homeomorphism of pairs\begin{equation}\label{Homeomorphism 1}f_K:(Z_{ K}, \mathcal{L}_G)\rightarrow (Z_{ U}, \mathcal{L}_G)\end{equation}that restricts to the identity in a neighborhood of $\mathcal{L}_G$ by the proof of Proposition \ref{Proposition Smooth Structures 1}. The 4-manifolds $Z_U^*$ and $Z_K^*$ are the result of surgeries along the components of the submanifold $\mathcal{L}_G\subset B_G$ and we define  $\Gamma\subset B_G^*\setminus \nu (T)$ to be the 2-link given by the disjoint union of the belt $2$-spheres of such surgeries. In particular, there is an embedding of $\Gamma$ inside $Z_K$ for any knot $K \in \mathcal{K}$, and the number of components of $\Gamma$ and of $\mathcal{L}_G$ are both equal to $n$. Moreover, the homeomorphism (\ref{Homeomorphism 1}) induces a homeomorphism of pairs\begin{equation*}g_K: (Z_K^*,\Gamma)\to(Z_U^*,\Gamma).\end{equation*} Proposition \ref{Proposition Diffeo Ambient} allows us to fix a diffeomorphism $\phi_U:Z_U^*\to M\# n (S^2\times S^2)$ and to define $\Gamma_U$ as $\phi_U(\Gamma)$.\\
Notice that we also have a diffeomorphism
\begin{equation}\label{eq: definition phi_K}\phi_K: Z_K^*\to M\# n (S^2\times S^2)\end{equation}such that $(\phi_K)_*=(\phi_U\circ g_K)_*$ in homology for every $K\in \mathcal{K}$. Indeed, any diffeomorphism taken from Proposition \ref{Proposition Diffeo Ambient} can be adjusted at the level of homology by composing it via a self-diffeomorphism of $M\# n(S^2\times S^2)$ by a result of Wall \cite[Theorem 2]{[Wall]} since the intersection form of $M$ is indefinite; see Lemma \ref{lemma inequality Betty number M}.\\
We define the 2-link (\ref{eq: LINKS GAMMA_K}) to be the image of $\Gamma_K$ under the diffeomorphism (\ref{eq: definition phi_K}), i.e. \begin{equation}\label{gammak} \Gamma_K := \phi_K(\Gamma).\end{equation}
We now argue that the 2-links $\Gamma_K$ and $\Gamma_U$ are topologically isotopic for any knot. The map\begin{equation*}h_K=\phi_U\circ g_K\circ\phi_K^{-1}:(M\# n(S^2\times S^2),\Gamma_K)\to (M\# n(S^2\times S^2),\Gamma_U)\end{equation*} is a homeomorphism of pairs that induces the identity map in homology. Work of Perron \cite{[Perron]} and Quinn \cite{[Quinn]} guarantees that $h_K$ is isotopic to the identity map and, hence, we conclude that the 2-links $\Gamma_K$ and $\Gamma_{U}$ are topologically isotopic for any $K\in\mathcal{K}$. Each component of the 2-link (\ref{eq: LINKS GAMMA_K}) is topologically unknotted by either \cite[Theorem B]{[Torres]} or \cite[Theorem 7.2]{[Sunukjian2]}. This concludes the proof of the first clause of Proposition \ref{Proposition Topologically Isotopic}.

The second clause is straightforward from the construction of $Z_K$ and $Z_K{^\ast}$.

\end{proof}

\subsection{Smoothly inequivalent $2$-links}\label{Section Smoothly Inequivalent} In this section, we distinguish the smooth embeddings of our 2-links. The key idea is to tell them apart by looking at the smooth structures of their complements and use their Seiberg-Witten invariants indirectly. More precisely, we undo the surgeries and reconstruct $Z_{K}$ from $M\#n(S^2\times S^2)$ as\begin{equation}\label{Reconstruction 2}Z_{K} \approx (M\#n(S^2\times S^2)\setminus  \nu(\Gamma_K))\cup (\nu(\mathcal{L}_G))\end{equation}and use inequivalent smooth structures constructed in Proposition \ref{Proposition Smooth Structures 1} in order to distinguish our 2-links.

\begin{proposition}\label{Proposition Smoothly Inequivalent links}Let $\Gamma$ and $\Gamma'$ be a pair of 2-links that are smoothly embedded in a smooth 4-manifold $Z^*$. Let $Z$ and $Z'$ be the 4-manifolds that are obtained from $Z^*$ by doing surgery along every component of $\Gamma$ and $\Gamma'$, respectively. If there is no diffeomorphism $Z\rightarrow Z'$, then the 2-links $\Gamma$ and $\Gamma'$ are smoothly inequivalent.

In particular, the 2-links (\ref{eq: LINKS GAMMA_K}) constructed in Proposition \ref{Proposition Topologically Isotopic} are pairwise smoothly inequivalent.
\end{proposition}

The proof of Proposition \ref{Proposition Smoothly Inequivalent links} is immediate and, hence, omitted. Any distinct pair of 2-links in the collection (\ref{eq: LINKS GAMMA_K}) satisfies the assumptions of this proposition by Proposition \ref{Proposition Smooth Structures 1}.

\subsection{Symmetric and Brunnianly exotic 2-links}\label{Section Brunnian}In this section, we study some properties of the 2-links of Theorem \ref{Theorem A}. We start with the following result regarding their symmetry.

\begin{proposition}\label{prop Symmetricity}
Suppose $G = F_g$. Every 2-link $\Gamma_K \subset M\#g(S^2\times S^2)$ belonging to the infinite collection (\ref{eq: LINKS GAMMA_K}) is smoothly symmetric.
\end{proposition}
\begin{proof}For any permutation $\sigma$ of $g$ elements, there is a self-diffeomorphism\begin{equation}\label{Self Diffeo}f: N_g\rightarrow N_g\end{equation}that is the identity on a neighborhood of the $2$-torus (\ref{T}) and such that it maps the framed loops (\ref{loops2}) to  $f(\gamma'_i)=\gamma'_{\sigma(i)}$ for every $i = 1,\dots, g$. The map (\ref{Self Diffeo}) yields a diffeomorphism of pairs\begin{equation*}
    (Z_K,\gamma'_1,\dots,\gamma'_g)\to(Z_K,\gamma'_{\sigma(1)},\dots,\gamma'_{\sigma(g)}),
\end{equation*}which allows us to define a diffeomorphism of pairs\begin{equation*}
    (Z^*_K,S_1,\dots,S_g)\to(Z^*_K,S_{\sigma(1)},\dots,S_{\sigma(g)}).
\end{equation*} In particular, the $2$-link $\Gamma:=S_1\sqcup\dots\sqcup S_g\subset Z_K^*$ is smoothly symmetric and so is the 2-link (\ref{gammak}) smoothly embedded in $M\# g(S^2\times S^2)$, given that there is a diffeomorphism of pairs\begin{equation*}\phi_K:(Z^*_K,\Gamma)\to(M\#g(S^2\times S^2),\Gamma_K);\end{equation*} see(\ref{eq: definition phi_K}).
\end{proof}

We will employ the following two results to address Brunnianity.

\begin{lemma}\label{gompf mandlebaum}Mandelbaum \cite{[Mandlebaum]}, Gompf \cite[Lemma 4]{[Gompf]}. 
    Let $X$ and $B$ be two oriented 4-manifolds (possibly with boundary) and let $T_X \subset X$ and $\alpha_B\times \beta_B=T_B \subset B$ be two smoothly embedded framed 2-tori. Suppose that $X$, $B$ and $X\setminus \nu(T_X)$ are simply connected and that either $X$ is spin, or $X\setminus \nu (T_X)$ is non-spin. Consider the generalized fiber sum\begin{equation*}F = X\#_{T^2}B\end{equation*} of $X$ and $B$ along $T_X$ and $T_B$. Then $F\# S^2 \times S^2$ is diffeomorphic to $X\# B^\ast$, where $B^\ast$ is the manifold obtained from $B$ by doing surgery along the push-offs of the loops $\alpha_B$ and $\beta_B$. Moreover, we can assume the chosen diffeomorphism to be the identity on $\partial F = \partial X \sqcup \partial B$ if $X$ or $B$ have non-empty boundary.
\end{lemma}

\begin{lemma}\label{spin states}
    Let $M$ be an admissible 4-manifold. If $M$ is non-spin, then, up to swapping the roles of the $2$-tori $T_1$ and $T_2$, $M\setminus \nu(T_2)$ is non-spin. 
\end{lemma}

\begin{proof} Since $M\setminus \nu(T_2)$ is simply connected, it is enough to show that there exists an element $\sigma \in H_2(M\setminus \nu(T_2); \Z)$ with odd self-intersection. The assumption $\pi_1(M\setminus (\nu(T_1) \sqcup \nu(T_2)))=\{1\}$ implies the existence of immersed 2-spheres $S_1,S_2 \subset M$ such that $S_i\cdot T_i=1$ for $i=1,2$ and $S_1 \cap T_2 = S_2 \cap T_1 = \emptyset$. If either of them has odd self-intersection, then the complement of one of the 2-tori contains a surface of odd self-intersection and we are done. Suppose that $S_1 \cdot S_1$ and $S_2\cdot S_2$ are both even. Since $M$ is non-spin, there exists an element $\alpha \in H_2(M; \Z)$ with odd self-intersection. Define $\sigma:=\alpha -( \alpha \cdot [T_2])[S_2]$ and notice that we have \begin{equation*}
        \sigma\cdot \sigma=\alpha\cdot \alpha-2(\alpha \cdot  [T_2])([S_2]\cdot \alpha) +(\alpha\cdot [T_2])^2 \alpha\cdot [T_2] \equiv \alpha \equiv 1 \ mod \ 2
    \end{equation*}
    and \begin{equation*}
        \sigma \cdot [T^2]=\alpha \cdot [T_2] - \alpha \cdot [T_2]=0.
    \end{equation*}
    Let $\Sigma \subset M$ be an embedded surface representing the homology class $\sigma \in H_2(M;\Z)$. The algebraic intersection of $\Sigma$ and $T_2$ is zero. We can eliminate the geometric intersection by tubing $\Sigma$ in oppositely signed intersection points, obtaining a new surface $\Sigma'$ in $M\setminus \nu(T_2)$ which has odd self-intersection.
\end{proof}
We are now ready to prove the main result of this section.
\begin{proposition}\label{Brunnian}
    For $G=F_g$, consider the family (\ref{eq: LINKS GAMMA_K}). There exists an infinite subset $\mathcal{K}' \subset \mathcal{K}$ of knots $K \subset S^3$ parametrizing a subfamily 
    \begin{equation}\label{K' family}
        \{\Gamma_K \mid K \in \mathcal{K}'\}
    \end{equation} of pairwise Brunnianly exotic $2$-links. 

\end{proposition}
\begin{proof}
        Let $\Gamma_K = S_{K,1} \sqcup S_{K,2} \sqcup ... \sqcup S_{K,g}$ be a g-component 2-link from the family (\ref{eq: LINKS GAMMA_K}). Thanks to Proposition \ref{prop Symmetricity}, in order to prove the theorem it is enough to show that removing the first component from each $\Gamma_K$ yields a family of 2-links that is no longer exotic. By construction, we have the following diffeomorphism of pairs 
    \begin{equation}\label{brunnian decomposition}
        (M\# n(S^2 \times S^2) , \Gamma_K) \approx (M_K\#_{T^2_K} (N^* \#_{T^2} ... \#_{T^2} N^*),\Gamma).
    \end{equation}
    Here the notation $\#_{T^2_K}$ is used to point out the fact that the framing of that generalized fiber sum depends on the knot $K$, while $\Gamma=S_1 \sqcup \dots \sqcup S_g$ is the 2-link made up by the belt spheres of the surgery (\ref{Definition N_g*}) on $N_g$. Such diffeomorphism can be chosen to send the component $S_{K,1}$ to $S_1$ for each knot $K$. In particular, each one of the components $S_i$ is contained in a $ N^* \approx (T^2 \times S^2) \#(S^2 \times S^2)$ block, which is therefore disjoint from the $(n-1)$-component $2$-link $\Gamma \setminus S_i $. In the following, we will use the copy of $S^2\times S^2$ that intersects $S_1$ to stabilize $\Gamma \setminus S_1$, which is contained in the remaining $N^*\#_{T^2} ... \#_{T^2} N^* \approx N_{g-1}^*$. In particular, we  have the following diffeomorphism of couples
\[(M\#g(S^2\times S^2),\Gamma\setminus S_{K,1})\approx ((M_K\#_{T^2_K} N^*) \#_{T^2} N_{g-1}^*,\Gamma\setminus S_1).\]

    To conclude it is enough to show that there exists an infinite subfamily $\mathcal{K'}\subseteq \mathcal{K}$ such that for any two knots $K,K'\in\mathcal{K'}$ there exists a diffeomorphism between $M_K\#_{T^2_K}N^*$ and  $M_{K'}\#_{T^2_{K'}}N^*$ relative to $\nu (T')$, where $T'$ is a parallel push-off of $T$ in $N^*$. In particular, we will see that there are at most four diffeomorphism types of pairs in the family $\{(M_K\#_{T^2_K}N^*,T') \mid K \subset \mathcal{K} \}$ where $T'$ is framed. By the pigeonhole principle, we can obtain $\mathcal{K'}$.

    We have the following identification by (\ref{Diffeomorphism 1})
    \begin{align*}
    M_K\#_{T^2_K}N^*  &\approx (M_K \# (S^2 \times S^2)) \#_{T^2_K} (T^2 \times S^2 )   
    \end{align*}
    in which the framed 2-torus $T'$ is sent to the canonically framed 2-torus $T^2\times\{p\}\subset T^2\times S^2$.

If $M$ is non-spin, then $M_K$ is non-spin and, by Lemma \ref{spin states}, $M \setminus \nu(T_2)$  is also non-spin, and they are all simply connected. Since $M_K$ is the generalized fiber sum between $M$ and $S^1 \times S^3$, we can apply Lemma \ref{gompf mandlebaum} to obtain a diffeomorphism $\psi$ between $(M\setminus \nu (T_2)) \# (S^2\times S^2)$ and $(M_K\setminus\nu (T_2)) \# (S^2\times S^2) $ which is the identity on the boundary $\partial \nu (T_2)$, see \cite{[Baykur]} for more details. We extend $\psi$ by the identity on $(T^2 \times S^2) \setminus \nu(T^2 \times \{p\})$, the result is the map $\Tilde{\psi}$
  \begin{align*}(M_K \# (S^2 \times S^2)) \#_{T^2_K} (T^2 \times S^2 )   
        \overset{\Tilde{\psi}}{\approx}(M \# (S^2 \times S^2)) \#_{T^2_K} (T^2 \times S^2 )   .  
   \end{align*}
   
    We can apply Lemma \ref{gompf mandlebaum} a second time as follows. Consider the generalized fiber sum $\mathcal{F}_K$ between $M$ and $(T^2\times S^2) \setminus \nu(T^2\times \{p\})$ along $T_2$, with framing depending on $K$, and $T^2\times\{q\}$. 
   Lemma \ref{gompf mandlebaum} implies that $\mathcal{F}_K \# (S^2 \times S^2 )$ is diffeomorphic to
   $M\# ((T^2\times S^2) \setminus \nu(T^2\times \{p\}))^*$ by a diffeomorphism $\varphi$
   that keeps the boundary $\partial \nu(T^2 \times \{p\})$ fixed. 
     Here we use the notation $X^*$ to denote the result of two loop surgeries on a 4-manifold $X$ along the loops described by Lemma \ref{gompf mandlebaum}.
   Hence, we can extend $\varphi$ via the identity to a map
   \begin{align*}(M \# (S^2 \times S^2)) \#_{T^2_K} (T^2 \times S^2 )   
        \overset{\Tilde{\varphi}}{\approx}M \#  (T^2 \times S^2)^*.  
   \end{align*}

   We do not need to determine the result of the loop surgeries 
   $(T^2 \times S^2)^*$
   . Since the loops do not depend on the knot $K$, but their framings might, there are four possibilities at most for the framed pair 
    $((T^2 \times S^2)^*,\Tilde{\varphi}(T^2\times \{p\})$. 
\end{proof}

\begin{remark}\label{stabilization}
    By taking the connected sum of the ambient manifold with a copy of $S^2\times S^2$, the same proof of Proposition \ref{Brunnian} shows that the 2-links of the family (\ref{K' family}) (instead of their sublinks) become smoothly equivalent in the ambient 4-manifold $(M\# g( S^2 \times S^2))\#(S^2 \times S^2)$.
\end{remark}

\begin{remark}\label{subfamily}
    It is possible to prove that $\mathcal{K}'$ is equal to $\mathcal{K}$. In particular, it is possible to avoid the use the pigeonhole principle in the preceding proof by taking a closer look at the framing of the loops and our definition of the framing $\tau_K$ for the 2-torus $T_2$ inside $M_K$.
\end{remark}

\begin{remark}\label{Remark Trivial Situation} From a pair of exotic 2-spheres in a 4-manifold, one easily obtains exotic 2-links with a prescribed number of components and the free group with $g$ generators as $2$-link group by adding unknotted 2-spheres. The results in this section prove that this is not the case for the $2$-links constructed in Theorem \ref{Theorem A} with 2-link group $F_g$.

\end{remark}

In the case of $G=\pi_1(\Sigma_g)$ and with the same arguments, one can show a weaker property after pairing up the 2-link components, in which disregarding a pair of components from the $n$-component 2-link yields smoothly equivalent $(n-2)$-component 2-links. 

\section{Proofs}
\subsection{Proof of Theorem \ref{Theorem A}}\label{Section Proof of Theorem A}

We collect the results of previous sections into a proof of Theorem \ref{Theorem A}. The infinite set (\ref{Unlink Main}) of 2-links $\{\Gamma_K \mid K\in\mathcal{K}\}$ smoothly embedded in $M\# n(S^2\times S^2)$ was constructed in Proposition  \ref{Proposition Topologically Isotopic}. The 2-link group of $\Gamma_K$ is isomorphic to the fundamental group of the manifold obtained from $M\# n(S^2\times S^2)$ by doing surgery on every component of $\Gamma_K$. This manifold is $Z_K$ by Proposition \ref{Proposition Topologically Isotopic} and by Proposition \ref{Proposition Smooth Structures 1} we know that its fundamental group $\pi_1(Z_K)$ is isomorphic to $G$. This settles the first clause of the theorem. 
The second clause of the theorem is that the collection of 2-links $\{\Gamma_K \mid K\in\mathcal{K}\}$ is an exotic family. By Proposition \ref{Proposition Topologically Isotopic}, elements of this family are pairwise topologically isotopic and by Proposition \ref{Proposition Smoothly Inequivalent links} they are pairwise smoothly inequivalent. The 2-links are componentwise topologically unknotted by Proposition \ref{Proposition Topologically Isotopic} and this establishes the third clause. The infinite set of pairwise non-diffeomorphic closed 4-manifolds of the fourth clause has been described in detail in Section \ref{Section Smooth Structures}. 
\hfill $\square$

\subsection{Proof of Theorem \ref{Theorem B}}\label{Section Proof of Theorem B} The 2-links of Theorem \ref{Theorem A} with $2$-link group $F_g$ are symmetric by Proposition \ref{prop Symmetricity}. The existence of the infinite subfamily of pairwise Brunnianly exotic 2-links that stabilize after one connected sum with $S^2\times S^2$ follows from Proposition \ref{Brunnian} and Remark \ref{stabilization}. 
\hfill $\square$



\end{document}